\documentclass[ a4paper,10pt]{amsart}
\usepackage[foot]{amsaddr}

\usepackage[T1]{fontenc}

\usepackage[utf8]{inputenc}

\usepackage{lmodern}

\usepackage[english]{babel}

\usepackage{amsmath}

\usepackage{amssymb}

\usepackage{mathtools}

\usepackage{mathrsfs}

\usepackage{xypic}

\usepackage{tikz-cd}

\usepackage{bookmark}

\usepackage{hyperref} 

\usepackage[alphabetic]{amsrefs}

\usepackage{enumitem}
\usepackage{graphicx}
\usepackage{tikz}
\usepackage{tikz-cd}
\usetikzlibrary{matrix,arrows,decorations.pathmorphing}
\usepackage[all,cmtip]{xy}

\newtheorem*{thm-no-num}{Theorem}
\newtheorem*{df-no-num}{Definition}
\newtheorem{thm}{Theorem} [section]
\newtheorem{maintheorem}{\bf Main Theorem}
\newtheorem{prop}[thm]{Proposition} 
\newtheorem{lm}[thm]{Lemma} 
\newtheorem{cor}[thm]{Corollary} 
\theoremstyle{remark}
\newtheorem{rmk}[thm]{Remark}

\theoremstyle{definition} 
\newtheorem {df}[thm]{Definition}

\newcommand{\bA}{\mathbb{A}}
\newcommand{\Brm}{{\rm B}}
\newcommand{\PP}{\mathbb{P}}
\newcommand{\ZZ}{\mathbb{Z}}

\newcommand{\OO}{\mathcal{O}}
\newcommand{\bQ}{\mathbb{Q}}
\newcommand{\cl}[1]{\mathcal{#1}}
\newcommand*{\sheafhom}{\mathcal{H}\kern -.5pt om}

\newcommand{\Mcal}{\cl{M}}

\newcommand{\Hcal}{\cl{H}}
\newcommand{\Lcal}{\cl{L}}
\newcommand{\Xcal}{\cl{X}}
\newcommand{\Ycal}{\cl{Y}}

\newcommand{\Gm}{\mathbb{G}_m}

\newcommand{\PGLt}{\textnormal{PGL}_2}

\newcommand{\pr}{{\rm pr}}

\renewcommand{\H}{{\rm H}}
\newcommand{\K}{{\rm K}}
\newcommand{\M}{{\rm M}}
\newcommand{\N}{{\rm N}}

\newcommand{\Het}{{\rm H}^{\bullet}}
\newcommand{\Inv}{{\rm Inv}^{\bullet}}
\newcommand{\Spec}{{\rm Spec}}

\newcommand{\del}{\partial}
\begin{document}
	\title[Brauer groups of $\Hcal_g$]{Brauer groups of moduli of hyperelliptic curves via cohomological invariants}
	\author[A. Di Lorenzo]{Andrea Di Lorenzo}
	\address{Aarhus University, Ny Munkegade 118, DK-8000 Aarhus C, Denmark}
	\email{andrea.dilorenzo@math.au.dk}
	\author[R. Pirisi]{Roberto Pirisi}
	\address{KTH Royal Institute of Technology, Lindstedtsvägen 25, 10044 Stockholm, Sweden}
	\email{pirisi@kth.se}	
	\date{\today}
\begin{abstract}
We compute the Brauer group of the moduli stack of hyperelliptic curves $\Hcal_g$ over any field of characteristic zero. In positive characteristic, we compute the part of the Brauer group whose order is prime to the characteristic of the base field.

\end{abstract}
\maketitle
\section*{Introduction}

\subsection*{Brauer groups of moduli stacks} 
Brauer groups of fields have long been an object of study in number theory, dating back to work of Noether and Brauer. They were later generalized by Grothendieck to schemes and more general objects, up to the vast generality of topoi. 

While Brauer groups of schemes have seen a lot of attention in modern algebraic geometry, computations of Brauer groups of moduli stacks over non-algebraically closed fields only started appearing in recent years.

In the $2010$s, Lieblich \cite{Lie} computed the Brauer group of ${\rm B}\mu_q$ over fields where $q$ is invertible and applied it to the period-index problem. Later, Antieau and Meier \cite{AntMeiEll} computed the Brauer group of the moduli stack $\Mcal_{1,1}$ of elliptic curves over a variety of bases, including $\ZZ$, $\bQ$ and any finite field of characteristic greater than two. Moreover, in an unpublished draft \cite{Mei} Meier computes the $\ell$-torsion of the Brauer group of $\Mcal_{1,1}$ over any noetherian scheme $S$ where $\ell$ is invertible.



In 2019, Shin \cite{Shi} showed that over an algebraically closed field of characteristic $2$ the Brauer group of $\Mcal_{1,1}$ is equal to $\ZZ/2\ZZ$.

The proofs of these results are based on standard tools in \'etale and flat cohomology coupled with a very delicate analysis of various presentations of the stack of elliptic curves, their relations, the stabilizer groups at various points, etc., which seem hard to apply to more complicated stacks.

\subsection*{Main results}
Our goal is twofold: 
\begin{enumerate}
    \item \emph{introducing a new toolkit for the computation of Brauer groups of moduli stacks}, based on the second author's theory of cohomological invariants for algebraic stacks \cite{PirAlgStack}. \\
    \item \emph{computing the prime-to-char(k) part of} ${\rm Br}(\Hcal_g)$, the Brauer group of the moduli stacks of hyperelliptic curves of genus $g\geq 2$, over a field $k$ of characteristic different from $2$ (see Main Theorem \ref{mainthm} below). In particular if ${\rm char}(k)=0$ we obtain the full Brauer group.
\end{enumerate}


To simplify our statements, we introduce the following notation:

\begin{df-no-num}
Let $A$ be an abelian torsion group, and let $c$ be a prime number or zero. We denote by $^cA$ the subgroup of $A$ given by elements whose order is not divisible by $c$. In particular $^0\!A=A$. Given a positive number $\ell$, we denote by $A_{\ell}$ the $\ell$-tosion subgroup.

  
\end{df-no-num}

Without further ado, we state our main Theorem (Thm. \ref{thm:BrauerHg}):

\begin{maintheorem}\label{mainthm}
Let $k$ be a field of characteristic $c\neq 2$ and $g>1$ an integer. Let $r_g \in \lbrace 0, 1 \rbrace$ be the remainder of $g$ mod $2$. Let $\ell_g$ be equal $2^{r_g}(4g+2)$.

Then
\[^c{\rm Br}(\Hcal_g)\simeq {^c{\rm Br}(k)} \oplus {\rm H}^1_{\textnormal{Gal}}(k, {^c}(\ZZ/\ell_g\ZZ)) \oplus \ZZ/2\ZZ^{\oplus (1+r_g)} .\]

\end{maintheorem}

Note that while in Antieau and Meier's results all non-trivial classes in the Brauer group come from cyclic algebras, in our case the $\ZZ/2\ZZ^{\oplus(1+r_g)}$ component does not, painting a richer picture. The additional copy of $\ZZ/2\ZZ$ in the odd case is generated by the class of the relative Brauer--Severi scheme $P=\mathcal{C}/\iota$, where $\mathcal{C}$ is the universal curve over $\Hcal_g$ and $\iota$ is the hyperelliptic involution. It is well known that this class is trivial when $g$ is even.

As we mentioned earlier, our techniques are based on the second author's theory of cohomological invariants for algebraic stacks \cite{PirAlgStack} and computations in Rost's (equivariant) Chow groups with coefficients \cites{Rost, Guil}. We believe these techniques will be well suited to compute the Brauer groups of a large variety of stacks admitting a ``good'' quotient presentation, such as the stacks of trigonal curves \cite{BolVis} and stacks that can be obtained as stacks of complete intersections, e.g. the stacks of quasi-polarized $K3$ surfaces of low degree \cite{DilK3}.

In terms of computing cohomological invariants, our main result (combining Prop. \ref{prop:CohInvHgodd}, Thm. \ref{thm:CohInvHg2} and Rmk. \ref{rmk:last inv}) is the following:

\begin{maintheorem}\label{main:Invariants}
Let $k$ be a field of characteristic $c \neq 2$, and let $\M$ be an $\ell$-torsion module, with $c \nmid \ell$. There are elements $\alpha_1,\ldots,\alpha_{g+1}$, of degree $1,\ldots,g+1$, such that \[I_g=\M^{\bullet}(k) \oplus \alpha_1 \cdot \M^{\bullet}(k)_{\ell_g} \oplus \bigoplus_{i=2}^{g+1} \alpha_i \cdot \M^{\bullet}(k)_2 \]
is a submodule of $\mathrm{Inv}^{\bullet}(\mathcal{H}_g,\M)$. If $g$ is even, we have
\[\mathrm{Inv}^{\bullet}(\mathcal{H}^k_g,\M)=I_g \oplus \beta_{g+2}\cdot \M^{\bullet}(k)_2.\]
If $g$ is odd there is an exact sequence
\[0 \rightarrow I_g \oplus w_2\cdot \M^{\bullet}(k)_2 \rightarrow \mathrm{Inv}(\mathcal{H}^k_g,\M) \rightarrow {\rm N}^{\bullet}_g(k) \rightarrow 0\]
where $w_2$ has degree $2$, ${\rm N}^{\bullet}_g(k)\subseteq \M^{\bullet}(k)_2$ and the last map lowers degree by $g+2$.
\end{maintheorem}

The cohomological invariants in the Theorem have coefficients in a generic $\ell$-torsion cycle module $\M$, where $\ell$ is a positive integer not divided by $c$. This represents a large increase in generality compared to the authors' previous results. In \cite{PirCohHypEven, PirCohHypThree, DilCohHypOdd} the result above is obtained under the assumption that the base field is algebraically closed and the coefficients are taken in (twisted) étale cohomology with coefficients in $\ZZ/p\ZZ$, where $p$ is a prime number different from $c$. In \cite{DilPir} the authors managed to lift the condition that $k$ should be algebraically closed, but still used mod $p$ étale cohomology as coefficients.

A general $\ell$-torsion cycle module $\M$ is a module over the cycle module $\H_{\ZZ/\ell\ZZ}$ given by twisted mod $\ell$ étale cohomology, which has a natural graded-commutative ring structure. So the technical improvement between this result and the authors' previous results is composed of two steps: going from $\H_{\ZZ/p\ZZ}$ to $\H_{\ZZ/\ell\ZZ}$, and thus giving up the $\mathbb{F}_p$-vector space structure, and from $\H_{\ZZ/\ell\ZZ}$ to $\M$, and thus giving up the ring structure. The second step is the hardest, as the modules appearing are not very well behaved, not being free or faithful or even necessarily finitely generated. Nonetheless, as the formula in the Theorem above shows, they will still admit reasonable decompositions in terms of the action of $\H_{\ZZ/\ell\ZZ}$. 

Dealing with more general cycle modules is necessary to obtain Brauer groups as we need to consider twisted étale cohomology with coefficients in $\ZZ/\ell\ZZ(-1)=\mu_{\ell}^{\vee}$. We could restrict ourselves to just considering the cycle modules $\H_{D}$ coming from twisted étale cohomology with coefficients in a Galois module $D$, but extending the results to all torsion cycle modules (and in fact for much of the paper to all cycle modules) requires little extra effort. 

As part of our supporting results, we obtain a sharpening of recent results by Gille and Hirsch \cite{GilHir} on classical cohomological invariants (with generalized coefficients) which might be of interest by itself.

\subsection*{Outline of the paper}
In Section \ref{sec:preliminaries} we establish the basic results we will need for the rest of the paper on Brauer groups, cohomological invariants and Chow groups with coefficients. In particular in Subsection \ref{sec:CohInvBrauer} we prove that the cohomological Brauer group of a quotient stack is computed by cohomological invariants.

Section \ref{sec:m11} contains a first demonstration of our techniques. We extend a computation from \cite{PirAlgStack} to compute the cohomological invariants with general coefficients of the stack of elliptic curves $\Mcal_{1,1}$ and use it to compute its Brauer group, partially retrieving Antieau and Meier's result \cite{AntMeiEll}.

In Section \ref{sec:equiv} we compute the equivariant Chow groups with coefficients of various classifying stacks ${\rm B}G$, which will be used in our main computation.

Section \ref{sec: Sn PGL2} is dedicated to computing the generalized cohomological invariants of ${\rm BS}_n$ and ${\rm BPGL}_2$. We use Gille and Hirsch's splitting principle \cite{GilHir}, complementing it with some equivariant computations which show that in every case where the splitting principle applies, any non-trivial cohomological invariant is of $2$-torsion.

In Section \ref{sec: pres Hg} we describe a presentation of $\Hcal_g$ by Arsie and Vistoli \cite{ArsVis} and restate some results on the relation between the cohomological invariants of $\Hcal_g$ and ${\rm BS}_{2g+2}$ from \cite{DilPir}.

Section \ref{sec: Coh Inv Hg} is where we put all the results together to obtain our computation of the cohomological invariants of $\Hcal_g$.

Finally, in Section \ref{sec: Br Hg} we specialize the computation of cohomological invariants to obtain a presentation of the Brauer group and we describe each generator.




\subsection*{Acknowledgements} The idea for this paper originated by a question posed to the second author by R. Fringuelli. We thank him for asking the right question at the right time. We are thankful to David Rydh for helpful discussions and a very careful reading of a preliminary draft of the paper, and to Burt Totaro, Zinovy Reichstein and Angelo Vistoli for helpful suggestions.

\subsection*{Notation} We work over a base field $k$ of characteristic $c \neq 2$. The notation $\ell$ will be reserved for a positive integer, not necessarily prime, that is not divisible by $c$.

Every scheme and algebraic (also known as Artin) stack is assumed to be of finite type over $\Spec(k)$. By a Galois module over $k$ we always mean a locally constant sheaf of abelian groups on the small \'etale site of $\Spec(k)$.

Unless otherwise stated, by ${\rm H}^{i}(X,F)$ we always mean \'etale cohomology, or lisse-étale cohomology for algebraic stacks. If $R$ is a $k$-algebra, we will write ${\rm H}^{i}(R,F)$ for ${\rm H}^{i}(\Spec(R),F)$. Given a graded abelian group $A$, we denote by $A_\ell$ the $\ell$-torsion subgroup and by $A\!\left[d\right]$ the group shifted in degree by $d$.

\section{Preliminaries}\label{sec:preliminaries}


\subsection{Brauer group, cohomological Brauer group, cyclic algebras}

Given a Noetherian scheme $X$, the \emph{Brauer group} ${\rm Br}(X)$ is the group of Azumaya algebras over $X$, i.e.\ sheaves of unitary algebras which are \'etale locally isomorphic to the endomorphism group of a vector bundle over $X$, modulo the relation that $\mathcal{E} \sim \mathcal{E}'$ if there exist vector bundles $V$ and $V'$ such that $\mathcal{E} \otimes {\rm End}(V) \simeq \mathcal{E}' \otimes {\rm End}(V')$. This relation corresponds to Morita equivalence, and the group operation is given by tensor product.

The rank of an Azumaya algebra is always a square $n^2$, and algebras of rank $n^2$ are classified by $P{\rm GL}_n$-torsors, with the trivial ones coming from ${\rm GL}_n$-torsors. We have an exact sequence
\[1 \rightarrow \Gm \rightarrow {\rm GL}_n \rightarrow {\rm PGL}_n \rightarrow 1\]
which induces an exact sequence 
\[{\rm H}^1_{\textnormal{\'et}}(X, {\rm GL}_n) \rightarrow {\rm H}^1_{\textnormal{\'et}}(X, {\rm PGL}_n) \rightarrow {\rm H}^2_{\textnormal{\'et}}(X, \Gm).\]
The class of an Azumaya algebra always maps to a torsion element of ${\rm H}^2_{\textnormal{\'et}}(X, \Gm)$, and the map is injective, so that ${\rm Br}(X) \subseteq {\rm H}^2(X, \Gm)_{\rm tor}$. We call the torsion subgroup ${\rm H}^2_{\textnormal{\'et}}(X, \Gm)_{\rm tor}$ the \emph{cohomological Brauer group} of $X$, denoted ${\rm Br}'(X)$. In the setting of schemes, due to results of Gabber and de Jong \cites{Gab81, DJ}, we know that ${\rm Br}(X)={\rm Br'}(X)$ whenever $X$ carries an ample line bundle.

The definition of the Brauer group can be vastly extended. Given an algebraic stack, one can define the Brauer group ${\rm Br}(\Xcal)$, and if $\Xcal$ is quasi-compact or connected the inclusion ${\rm Br}(\Xcal) \subseteq {\rm Br}'(\Xcal)$ holds (note that for algebraic stacks we will have to use Lisse-\'etale cohomology). 

The cohomological Brauer group is often easier to compute, and in all of our computations we will work with it and then check a posteriori that every element we find comes from an Azumaya algebra. One very important type of elements that we know always come from the Brauer groups are those given by \emph{cyclic algebras}.

First, note that if $\ell$ is prime to ${\rm char}(k)$, and $\Xcal$ is an algebraic stack, due to the Kummer exact sequence 
\[1 \rightarrow \mu_\ell \rightarrow \Gm \rightarrow \Gm \rightarrow 1\]
the $\ell$-torsion of the ${\rm Br}'(\Xcal)$ is the image of ${\rm H}^2_{\textnormal{lis-\'et}}(\Xcal, \mu_\ell)$. Now, given elements 
\[\alpha \in {\rm H}^1_{\textnormal{lis-\'et}}(\Xcal, \mu_\ell),\quad \beta \in {\rm H}^1_{\textnormal{lis-\'et}}(\Xcal, \ZZ/\ell\ZZ)\]
there exists a canonical Azumaya algebra $\mathcal{A}_{\alpha, \beta}$ whose class in ${\rm H}^2_{\textnormal{lis-\'et}}(\Xcal, \mu_\ell)$ is equal to the cup product 
\[\alpha \cdot \beta \in {\rm H}^2_{\textnormal{lis-\'et}}(\Xcal, \mu_\ell) = {\rm Br}'(\Xcal)_\ell.\]

In particular, any element of ${\rm Br}'$ that can be written as above automatically belongs to the Brauer group.

The equality ${\rm Br}(\Xcal)={\rm Br}'(\Xcal)$, at least for the prime-to-${\rm char}(k)$ part, is known to hold for a large class of Deligne--Mumford stacks: the following theorem is a combination of a result by Edidin, Hassett, Kresch and Vistoli, \cite{EHKV}*{Thm. 3.6}\footnote{The theorem is stated for Noetherian schemes, but the key lemmas needed for the proof, in Section 2, are all formulated in terms of DM stacks and the proof of the theorem carries word by word for a separated Noetherian DM stack.}, which says that an element $\alpha \in {\rm Br}'(\Xcal)$ comes from the Brauer group if and only if the corresponding $\mu_n$-gerbe is a quotient stack, and a result by Kresch and Vistoli, \cite{KV}*{Thm. 2.2}, which gives a criterion for a Deligne--Mumford stack to be a quotient stack.

\begin{thm}[EHKV]\label{thm:EHKV}
Let $\Xcal$ be a smooth, separated, generically tame DM stack of finite type over $k$, and assume that $\Xcal$ has a quasi-projective coarse moduli space. Then we have $^c{\rm Br}(\Xcal)=\hspace{1pt}^c{\rm Br}'(\Xcal)$.
\end{thm}

In particular, the hypotheses of the Theorem are valid for the stacks $\Hcal_g$ of hyperelliptic curves and for the stack $\Mcal_{1,1}$ of elliptic curves, so we will know a priori that all the elements we produce belong to the Brauer group. We remark that the explicit descriptions of the elements we give would also be sufficient to show that they all belong to the Brauer group.

\subsection{Generalized cohomological invariants}


Classically, cohomological invariants are defined as natural transformations from the functor \[
{\rm T}_{G}:({\rm Field}/k) \to ({\rm Set}),\quad {\rm T}_{G}(F)\stackrel{\rm def}{=}\lbrace G\mbox{-torsors over}\,\, F/{\rm iso} \rbrace\]  which sends a field $F$ to the isomorphism classes of $G$-torsors over $F$, to the twisted cohomology functor \[{\rm H}_{D}:({\rm Field}/k) \to ({\rm Set}),\quad {\rm H}^{\bullet}_{D}(F)\stackrel{\rm def}{=}\oplus_i{\rm H}^i_{\rm Gal}(F,D(i))\] where $D$ is a torsion Galois module. This is the definition in Garibaldi, Merkurjev and Serre's book \cite{GMS}.

In \cite{PirAlgStack}, the second author extended the notion of cohomological invariants to a theory of invariants for algebraic stacks, retrieving the classical theory when the stack is equal to the classying stack of $G$-torsors ${\rm B}G$.

In both the classical case and the more general version of cohomological invariants we use in this paper, one can take a more general approach regarding coefficients and define cohomological invariants as natural transformations to any of Rost's cyle modules, as defined in \cite{Rost}. 

These kind of generalized invariants were considered by both Guillot \cite{Guil}*{Sec. 6} and the second author \cite{PirAlgStack}*{Sec. 6}, but as recently pointed out by Gille and Hirsch \cite{GilHir}*{Sec. 3}, both approaches only work for cycle modules defined over every extension of the base field $k$, such as Galois cohomology and Milnor's $\K$-theory, while in general a cycle module $M$ is only defined over finitely generated extensions of $k$. 

In their paper Gille and Hirsch rework the foundations of the theory in this more general case, and prove a splitting formula for finite reflection groups and Weil groups. Following their approach, in this section we will give a corrected definition of generalized cohomological invariants of algebraic stacks and show that the usual properties apply. In Section \ref{sec: Sn PGL2}, we will use equivariant techniques to obtain a sharpening of their splitting formula.

Given a field $F$, denote by ${\rm T}(F^*)$ the tensor ring $\oplus_{i\geq 0} (F^*)^{\otimes_{\ZZ}i}$. Milnor's $\K$-theory ring $\K_{\rm M}^\bullet(F)$ is given by 
\[
\K_{\rm M}^\bullet(F)={\rm T}(F^*)/\lbrace a \otimes b \mid a, b \in F^*, \, a+b=1 \rbrace
\]
For a field extension $\phi:F\to F'$ there is a map $\phi^*:\K_{\rm M}^{\bullet}(F) \to \K_{\rm M}^{\bullet}(F')$ given by restriction of scalars. If the extension is finite we have a map $\phi_*:\K_{\rm M}^{\bullet}(F') \to\K_{\rm M}^{\bullet}(F) $ given by the norm map. These maps are functorial and compatible with each other. Moreover, we have a projection formula
\[\phi_* \circ \phi^* = \left[ F':F \right] {\rm Id_{\K_{\rm M}^{\bullet}(F)}}.  \]
If the extension $F'/F$ is purely inseparable, we moreover have the opposite equality 
\[\phi^* \circ \phi_* = \left[ F':F \right] {\rm Id_{\K_{\rm M}^{\bullet}(F')}}.  \]
We call a DVR $(R,v)$ \emph{geometric} if $R$ is a $k$-algebra and the transcendence degree of the quotient field $F_R$ is one higher than the transcendence degree of the residue field $F_v$. In this case we have a boundary map

\[\partial_v: \K_{\rm M}^\bullet(F_R) \to \K_{\rm M}^{\bullet-1}(F_v).\]

In degree $1$, this map is just the valuation $v$. Finally, let $(R,v)$ as above and let $\pi$ be a uniformizer for $v$. We define a map 
\[ s_v^{\pi}: \K_{\rm M}^{\bullet}(F_R) \to \K_{\rm M}^{\bullet}(F_v), \quad
s_v^{\pi}(x) = \partial_v(\pi \cdot x).
\]
If $\partial_v(x)=0$ the map does not depend on the choice of $\pi$ and we just denote it $s_v$.

Milnor's $\K$-theory is the basic, and most important, cycle module. In general, denote $\mathfrak{F}_k$ the subcategory of $({\rm Field}/k)$ given by finitely generated extensions. A cycle module $M$ is a contravariant functor\footnote{We choose to see cycle modules as contravariant functors from the opposite category, contrary to what Rost does, so that the pullback/pushforward notation agrees with the one for Chow groups with coefficients.} from $\mathfrak{F}_k^{\rm op}$ to the category of graded abelian groups such that $M^{\bullet}(F)$ is a $\K^{\bullet}_{\rm M}(F)$-module for all $F$, and which satisfies a long list of properties, such as having the four operations $\phi_*, \phi^*, \partial_v, s_{v}^{\pi}$, these operations being compatible with the ones on $\K_{\rm M}$ and with each other, having a projection formula, etc \cite{Rost}*{Sec. 1-2}.

The key examples of cycle modules other than Milnor's $\K$-theory are given by the twisted Galois cohomology cycle modules ${\rm H}_D$. Note that by Voevodsky's Norm-Residue isomorphism \cite{Voe}*{Thm. 6.1} we have the equality \[\K_{\rm M}\!/(\ell) ={\rm H}_{\ZZ/\ell \ZZ}\stackrel{\rm def}{=} \K_{\ell}.\]

Given a geometric DVR $(R,v)$, Rost defines a group ${\rm M}^\bullet(v)$ (not to be confused with ${\rm M}^{\bullet}(F_v)$) with a map $p:{\rm M}^{\bullet}(F_R) \rightarrow {\rm M}^{\bullet}(v)$. The group sits in the exact sequence
\[0 \to {\rm M}^{\bullet}(F_v) \xrightarrow{i} {\rm M}^{\bullet}(v) \xrightarrow{\partial} {\rm M}^{\bullet}(F_v) \to 0 \]

and the composition $\partial \circ p$ is equal to $\partial_v$.

When working with Galois cohomology, the group ${\rm H}_D^{\bullet}(v)$ is equal to  ${\rm H}^{\bullet}_D(F_{R^h})$, where $(R^{h},v)$ is the Henselization of $(R,v)$ \cite{GMS}*{7.10}. The definition of cohomological invariants of an algebraic stack \cite{PirAlgStack}*{Def. 2.2} includes a continuity condition which can be rephrased as stating that, for an Henselian DVR $R^h$, the value of $p(\alpha(F_{R^{h}}))$ should be equal to $i(\alpha(F_v))$. This suggests the following general definition:

\begin{df}
Let $\Xcal$ be an algebraic stack. A cohomological invariant with coefficients in ${\rm M}$ is a natural transformation
\[\alpha:{\rm Pt}_{\Mcal} \longrightarrow {\rm M}\]
such that, for any geometric DVR $(R,v)$ and any map from the spectrum of $R$ to $\Xcal$ we have
\[ p(\alpha(F_R)) = i(\alpha(F_v)) \in {\rm M}^{\bullet}(v). \]

The cohomological invariants with coefficients in $\M$ of $\Xcal$ form a graded group $\Inv(\Xcal,{\rm M})$. 

Given a morphism $f:\Ycal \to \Xcal$ we define a pullback $f^*:\Inv(\Xcal,{\rm M}) \to \Inv(\Ycal,{\rm M})$ by setting $f^*(\alpha(q))=\alpha(f(q))$.
\end{df}

Note that this definition differs from the one in \cite{PirAlgStack}*{Sec. 6} as it requires the continuity condition on all DVRs, rather than just Henselian ones. As will be evident in the rest of the section, the theory we obtain with this definition is exactly the same as what we obtain in \emph{Loc. Cit.}

\begin{rmk}
In our notation, we call ${\rm Inv}^i({\rm B}G, \M)$ the group denoted by ${\rm Inv}^i(G,\M)$ in \cite{GilHir}, and ${\rm Inv}^i({\rm B}G, {\rm H}_D)$ the group denoted by ${\rm Inv}^i(G,D)$ in \cite{GMS}. 
\end{rmk}

\begin{rmk}
Given an étale extension $(R',v')$ of $(R,v)$ with the same residue field $F_v=F=F_{v'}$ we have an isomorphism ${\rm M}^{\bullet}(v) \simeq {\rm M}^{\bullet}(v')$ and the continuity condition for $(R',v')$ is equivalent to the continuity condition for $(R,v)$. 

In other words, we can check the continuity condition on any \emph{Nisnevich neighbourhood} of the closed point of $\Spec(R)$.

Moreover, if $M$ is defined on all extensions of $k$ we can restrict to checking the condition on the Henselization $R^h$.
\end{rmk}

Recall that a smooth-Nisnevich morphism \cite{PirAlgStack}*{Def. 3.2} is a representable smooth morphism $f:\mathcal{Y} \to \Xcal$ of algebraic stacks such that for any point $p:\Spec(F) \to \Xcal$ we have a lifting
\[
\xymatrix{  & \mathcal{Y} \ar[d]^f \\
        \Spec(F) \ar[ur]^{p'} \ar[r]^{p} & \Xcal  }
\]
A smooth $\ell$-Nisnevich morphism  \cite{PirAlgStack}*{Def. 3.4} is a representable smooth morphism $f:\mathcal{Y} \to \Xcal$ of algebraic stacks such that for any point $p:\Spec(F) \to \Xcal$ we have a commutative square
\[
\xymatrix{ \Spec(A) \ar[r]^{p'} \ar[d]^{\phi} & \mathcal{Y} \ar[d]^f \\
        \Spec(F)  \ar[r]^{p} & \Xcal  }
\]
where $A=F_1 \times \ldots \times F_r$, for each $i$ the extension $F_i/F$ is finite and separable and $(\left[F_1:F\right],\ldots,\left[F_r:F\right],\ell)=1$. 

Given an algebraic stack $\Xcal$ we define the smooth-Nisevich and smooth $\ell$-Nisnevich sites of $\Xcal$ as the sites $({\rm Spc}/\Xcal)_{\rm sm-Nis}$ and $({\rm Spc}/\Xcal)_{\rm sm \,\,\ell-Nis}$ where the objects are representable maps to $\Xcal$, morphisms are commutative squares over the identity of $\Xcal$ and coverings are respectively smooth-Nisnevich and smooth $\ell$-Nisnevich morphisms \cite{PirAlgStack}*{Def. 3.9}.

Let $A$ be as above. Then the fibered product $\Spec(A)\times_{\Spec(F)} \Spec(A)$ is a disjoint union $\Spec(E_1) \sqcup \ldots \sqcup \Spec(E_n)$ where each of the $E_j$ is a finite separable extension of $F$. There are two pullbacks ${\rm Pr}_1^*, {\rm Pr}_2^*:\oplus_i M(F_i) \to \oplus_j {\rm M}(E_j)$. We have a sequence 
\[
0 \to {\rm M}(F) \to \oplus_i {\rm M}(F_i) \xrightarrow{{\rm Pr}_1^* - {\rm Pr}_2^*} \oplus_j {\rm M}(E_j)
\]
which is always left exact when ${\rm M}$ is of $\ell$-torsion. We say that ${\rm M}$ has property $(\mathcal{S}_{\ell})$ if ${\rm M}$ is of $\ell$-torsion and the sequence above is exact for any $F\in \mathfrak{F}_k$ and $A$ as above. This is in particular true for ${\rm M}={\rm H}_{D}$ where $D$ is a $\ell$-torsion Galois module, see \cite{PirAlgStack}*{Lm. 3.7}.

\begin{thm}\label{thm:M sheaf}
The functor $\Inv(-,{\rm M})$ is a smooth-Nisnevich sheaf.

If ${\rm M}$ is of $\ell$-torsion, then the pullback of $\Inv(-,{\rm M})$ through a smooth $\ell$-Nisnevich covering is injective. 

If moreover ${\rm M}$ has property $(\mathcal{S}_{\ell})$, then $\Inv(-,{\rm M})$ is a smooth $\ell$-Nisnevich sheaf.
\end{thm}
\begin{proof}
The proof is very similar to the proof of \cite{PirAlgStack}*{Thm. 3.8}. Let $\pi:\mathcal{Y} \to \mathcal{X}$ be a smooth Nisnevich covering. We define an invariant $\alpha$ by setting $\alpha(p)=\alpha(p')$ for any lifting $p'$ of $p$. This is clearly well defined and functorial thanks to the sheaf condition.

Now let $(R,v)$ be a geometric DVR, let $f:\Spec(R) \to \Xcal$ be a morphism. Let $R^h$ be the Henselization of $R$. The morphism $\mathcal{Y} \times_{\Xcal} \Spec(R^h) \rightarrow \Spec(R^h)$ always has a section. This implies that we can find some finite extension $(R',v')$ of $R$ with the same residue field $F_v$ and a commutative square
\[
\xymatrix{ \Spec(R') \ar[r]^{f'} \ar[d]^{\phi} & \mathcal{Y} \ar[d]^\pi \\
        \Spec(R)  \ar[r]^{f} & \Xcal  }
\]
The continuity condition for $(R',v')$ implies the continuity condition for $(R,v)$, so we have proven our claim in the smooth-Nisnevich case.

Injectivity in the smooth $\ell$-Nisnevich case is trivial as given an extension $F \subset F_1 \times \ldots \times F_r$ as above the map 
\[
\phi_1^* \times \ldots \times \phi_r^* : {\rm M}^{\bullet}(F) \to {\rm M}^{\bullet}(F_1) \oplus \ldots {\rm M}^{\bullet}(F_r)
\]
is injective. If condition $(\mathcal{S}_{\ell})$ holds, then we can proceed exactly as above to show that cohomological invariants form a smooth $\ell$-Nisnevich sheaf.
\end{proof}

As a direct consequence of the continuity condition, we see that given a geometric DVR $R$ the value of a cohomological invariant being zero at the quotient field implies the same for the residue field. Then it is easy to prove by induction that this property propagates to regular local rings and regular schemes of finite type.

\begin{lm}
Let $X$ be a regular, connected scheme of finite type over $k$, with generic point $\xi$. Then for any $\alpha \in \Inv(X,{\rm M})$ we have
\[ \alpha(\xi)=0 \Rightarrow \alpha=0. \]
\end{lm}
\begin{proof}
This is proven exactly as in \cite{GilHir}*{Cor. 3.6} or \cite{PirAlgStack}*{Lm. 4.5}.
\end{proof}

Given a smooth, irreducible scheme $X$ over $k$, the \emph{zero codimensional Chow group with coefficients} $A^0(X,{\rm M})$ is the subgroup of ${\rm M}(k(X))$ of elements such that $\partial_v a = 0$ whenever $(R,v)$ is the local ring of a point $x \in X^{(1)}$, i.e. a point of codimension one in $X$.

The following Theorem shows that on smooth schemes, this group is equal to the group of cohomological invariants, and consequently in general the functor of cohomological invariants is a sheafification of $A^{0}(-,{\rm M})$ in the smooth-Nisnevich (or smooth $\ell$-Nisnevich) topology. The properties of Chow groups will coefficients will be discussed in depth in subsection \ref{sec:Chow}.

\begin{thm}
Let $X$ be a scheme, smooth over $k$. Then
\[ \Inv(X,{\rm M}) = A^0(X,{\rm M}). \]

Let $\Xcal$ be an algebraic stack, smooth over $k$. Then the functor $\Inv(-,{\rm M})$ is the smooth-Nisnevich sheafification of $A^0(-,{\rm M})$.
\end{thm}
\begin{proof}
The second statement is an obvious consequence of the first and Theorem \ref{thm:M sheaf}.

To prove the first statement, let $X$ be a smooth, connected scheme over $k$ with generic point $\xi$. Note that given an invariant $\alpha$, the value $\alpha(\xi)$ is unramified at all points $x \in X^{(1)}$, so $\alpha(\xi) \in A^0(X,{\rm M})$. This map is injective by the Lemma above. On the other hand, an element $\beta \in A^0(X,{\rm M})$ defines a cohomological invariant by pullback (see \cite{Rost}*{p. 360, after Cor. 6.5} or \cite{Rost}*{12.2, 12.4}), and the two maps are inverse to each other.
\end{proof}

\begin{rmk}
When ${\rm M}={\rm H}_{D}$, by the Bloch-Ogus-Gabber exact sequence we have 
\[
\Inv(X,{\rm H}_{D})=A^0(X,{\rm H}_{D})={\rm H}^0_{\rm Zar}(X,{\rm H}^{\bullet}(-,D))
\]
that is, on smooth schemes cohomological invariants are the Zariski sheafification of the twisted cohomology functor ${\rm H}^{\bullet}(-,D)$. 
\end{rmk}

\begin{cor}\label{cor:hom inv}
Let $f:\mathcal{Y} \rightarrow \Xcal$ be map of smooth algebraic stacks over $k$. If $f$ is either:
\begin{itemize}
    \item An affine bundle.
    \item An open immersion whose complement has codimension at least $2$.
\end{itemize}
Then the pullback $f^*$ is an isomorphism.
\end{cor}
\begin{proof}
This is an easy consequence of the same properties being true for $A^0(X,{\rm M})$ (\cite{Rost}*{Sec. 9} for the first statement, and the second one is by definition) and the sheaf conditions.
\end{proof}

Finally, the following Proposition shows that when working with an $\ell$-torsion cycle module we can always reduce to $\ell$ being a prime power.

\begin{prop}\label{prop: l=rs}
Assume that $\ell= r s$, with $(r,s)=1$, and let $\M$ be an $\ell$-torsion cycle module. Then
\begin{enumerate}
    \item $\M_r$ and $\M_s$ are both cycle modules.
    \item We have $\M=\M_r \times \M_s$, and if $\M$ has a product pairing the product respects the ring structure.
    \item We have $\Inv(\Xcal,\M)= \Inv(\Xcal, \M_r) \oplus \Inv(\Xcal,\M_s)$.
\end{enumerate}
\end{prop}
\begin{proof}
The four operations $\phi^*, \phi_*, \partial_v, s_{v}^{\pi}$ all send $\M_r$ to itself. Verifying that the requirements in \cite{Rost}*{Sec. 1-2} are all satisfied is a routine exercise, proving $(1)$. Moreover, we have maps $\M \xrightarrow{\cdot s} \M_r, \, \M \xrightarrow{\cdot r} \M_s$. The resulting map $\M \xrightarrow{\rho} \M_r \times \M_s$ is clearly an isomorphism, proving $(2)$. Finally, the map $\rho$ induces a map at the level of cohomological invariants which again is easily shown to be an isomorphism, proving $(3)$.
\end{proof}

\subsection{Cohomological invariants and the cohomological Brauer group}\label{sec:CohInvBrauer}

In this Subsection we restate some results from the last section of \cite{PirAlgStack} connecting cohomological invariants and the cohomological Brauer group, and we extend them to quotient stacks. Recall that $\ell$ is a positive number not divisible by ${\rm char}(k)$.

\begin{lm}\label{InvBrauer}
Let $X$ be an algebraic space, quasi-separated and smooth over $k$. Then 
\[ {\rm Inv}^2(X,{\rm H}_{\mu_{\ell}^{\vee}}) = \mathrm{Br}'(X)_\ell. \]
\end{lm}
\begin{proof}
This is the content of \cite{PirAlgStack}*{Lem. 7.6}.
\end{proof}

The proof of this statement does not extend to algebraic stacks, but as for cohomological invariants, any theory that is invariant by the operations of removing closed subsets of high codimension and of passing to vector bundles admits an equivariant extension to quotient stacks. The following Lemma tells us exactly that.

\begin{lm}
Let $\Xcal$ be an algebraic stack smooth over $k$. Then
\begin{itemize}
\item If $\mathcal{V} \rightarrow \Xcal$ is a vector bundle we have $\mathrm{Br}'(\Xcal)_\ell=\mathrm{Br}'(\mathcal{V})_\ell$.
\item If $\mathcal{U} \subset \Xcal$ is an open  subset whose complement has codimension $\geq 2$ we have $\mathrm{Br}'(\Xcal)_\ell=\mathrm{Br}'(\mathcal{U})_\ell$. 
\end{itemize}
\end{lm}
\begin{proof}
This is proven in \cite{FriPirBrauer}*{Prop. 1.3, 1.4}.
\end{proof}

Now we can use equivariant approximation to reduce the problem of computing the cohomological Brauer group to algebraic spaces.

\begin{prop}\label{prop:Inv to Br}
Let $X$ be an algebraic space, quasi-separated and smooth over $k$, being acted upon by an affine smooth group scheme $G/k$. Then
\[ {\rm Inv}^2(\left[ X/G\right],{\rm H}_{\mu_{\ell}^{\vee}}) = \mathrm{Br}'(\left[ X/G\right])_\ell. \]
\end{prop}
\begin{proof}
Pick a representation $V$ of $G$ such that the action is free on an open subset $U$ whose complement has high codimension. Such a representation can always be found, as any smooth affine algebraic group over $k$ admits an embedding into ${\rm GL}_n$. Then the equivariant approximation $\left[U \times X /G\right]$ has the same cohomological Brauer group and cohomological invariants as $\left[X/G\right]$. Lemma \ref{InvBrauer} allows us to conclude.
\end{proof}

Finally, the following Lemmas will be useful in detecting cyclic algebras among elements of the Brauer groups.

\begin{lm}\label{lm:H1}
Let $X/k$ be a smooth and separated algebraic space over $k$, being acted upon properly by an affine smooth group scheme $G/k$. Then
\[ {\rm Inv}^1(\left[ X/G \right], \K_{\ell}) = {\rm H}^1(\left[ X/G \right], \mu_\ell).\]
\end{lm}
\begin{proof}
Note that by purity we can remove closed subsets of codimension $\geq 2$ without affecting the first cohomology groups. By equivariant approximation it is then sufficient to prove the statement for a smooth separated algebraic space $X$. 

Moreover, by \cite{StPr}*{Tag 0ADD} we know that, up to a closed subset of codimension at most $2$, $X$ is a scheme, so we can reduce to a smooth scheme $X$.

Consider the morphism of sites \[(i_*, i^*): X_{\textnormal{\'et}} \to X_{\textnormal{Zar}}\] given by restriction and pullback. The composition with the global sections functor $\Gamma$ induces a Grothendieck spectral sequence
\[ {\rm H}^p_{\textnormal{Zar}}(X,{\rm R}^qi_*F) \Rightarrow {\rm H}^{p+q}_{\textnormal{\'et}}(X,F) .\]
Picking $F=\mu_\ell$, the low degree terms exact sequence reads
\[ 0 \rightarrow  {\rm H}^1_{\textnormal{Zar}}(X,\mu_\ell) \rightarrow {\rm H}^1_{\textnormal{\'et}}(X,\mu_\ell) \rightarrow {\rm Inv}^1(X, \K_{\ell}) \rightarrow {\rm H}^2_{\textnormal{Zar}}(X,\mu_\ell). \]
Now observe that on the small Zariski site of $X$, the sheaf $\mu_\ell$ is constant, and thus flasque. Then its Zariski cohomology is trivial and the result follows immediately.
\end{proof}

\begin{lm}\label{lm:cyclic}
Let $\Xcal$ be an algebraic stack over $k$. Any element of ${\rm Br}'(\Xcal)$ coming from the cup product ${\rm H}^1(X, \mu_\ell) \otimes {\rm H}^1(\Xcal, \ZZ/\ell\ZZ) \rightarrow {\rm H}^2(\Xcal, \mu_\ell)$ is represented by a cyclic algebra.
\end{lm}
\begin{proof}
This is proven in \cite{AntMeiEll}*{2.10, 3.7}.
\end{proof}

\subsection{Chow groups with coefficients}\label{sec:Chow}

As our main tool for computations, we need to establish the main properties of the theory of Chow groups with coefficients. Rost's original paper \cite{Rost} develops the theory for a general cycle module, but it notably lacks a theory of Chern classes. This is developed for a general cycle module in \cite{PirCohHypEven}*{2.3-2.5}.

A cycle with coefficients $\alpha \in C_{i}(X,{\rm M})$ is a formal finite sum of elements in the form $(V, \tau)$, where $V$ is an irreducible subscheme of $X$ of dimension $i$ and $\tau \in {\rm M}^{\bullet}(k(V))$. The groups $C_i(X,{\rm M})$ form a complex 
\[0 \rightarrow C_{{\rm dim}(X)}(X,{\rm M}) \rightarrow \ldots \rightarrow  C_1(X,{\rm M}) \rightarrow C_0(X,{\rm M}) \rightarrow 0 \]
If $V$ is a normal irreducible subscheme, the differential ${\rm d}$ is defined on an element $(V,\tau)$ by
\[{\rm d}(V,\tau) = \oplus_{p \in V^{(1)}} (\overline{p},  \partial_{v_p} \tau)\] where the sum runs over points of codimension $1$ and the valuation $v_p$ is defined by the DVR $\mathcal{O}_{V,p}$. There is some extra subtlety involved in the general definition when the subscheme $V$ is not be normal.

The axioms of cycle modules ensure that the sum is finite and that ${\rm d}\circ {\rm d}=0$. The groups $A_i(X,{\rm M})$ are defined as the homology of this complex.

When $X$ is equidimensional, which will always be the case in the following, we can define $C^i(X,\M)$ as the group of cycles of codimension $i$, and switch to the codimension notation
 
\[ A^i(X,\M)=A_{{\rm dim}(X)-i}(X,\M) \]

which is better suited for our purposes. We denote \[A^*(X,\M)= \bigoplus_{i=0}^{{\rm dim}(X)} A^i(X,\M).\] The group $A^*(X,\M)$ has two different gradations, one given by codimension and one by the gradation on $\M$. We will always refer to the former by codimension and the latter by degree.
 
We recall the main properties we are interested in:

\begin{prop}\label{prop:properties}
	Let $X$ and $Y$ be equidimensional, quasi-projective schemes of finite type over $k$. Let $\M$ and $\N$ be $\ell$-torsion cycle modules. Then we have:
	\begin{enumerate}		
		\item \emph{Proper pushforward}: every proper morphism $f:X\to Y$ induces a homomorphism of groups
		\[f_*:A_i(X,\M)\longrightarrow A_i(Y,\M)\]
		which preserves the cohomological degree.
		
		\item \emph{Flat pullback}:  every flat morphism $f:X\to Y$ of relative constant dimension induces a homomorphism of groups
		\[f^*:A^i(Y,\M)\longrightarrow A^i(X,\M)\]
		which preserves the cohomological degree.
		
		\item \emph{Localization exact sequence}: given a closed subscheme $Z\xhookrightarrow{i} X$ whose open complement is $U\xhookrightarrow{j} X$, there exists a long exact sequence
		\[\cdots\to A_i(X,\M) \xrightarrow{j^*} A_i(U,\M) \xrightarrow{\del} A_{i-1}(Z,\M) \xrightarrow{i_*} A_{i-1}(X,\M) \to \cdots \]
		The boundary homomorphism $\del$ has cohomological degree $-1$, whereas the other homomorphisms have cohomological degree zero.
		
		\item \emph{Compatibility}: given a cartesian square of schemes
		\[ \xymatrix{
			Y \ar[r]^i \ar[d] & X \ar[d] \\
			Y' \ar[r]^{i'} & X' }\]
		where the vertical morphisms are closed emebeddings and the horizontal ones are proper, we get a commutative square
		\[ \xymatrix{
			A_k(Y'\setminus Y,\M) \ar[r]^{i''_*} \ar[d]^{\del} & A_k(X'\setminus X,\M) \ar[d]^{\del} \\
			A_{k-1}(Y,\M) \ar[r]^{i_*} & A_{k-1}(X,\M) } \]
			where $i''$ is the restriction of $i'$ to $Y'\setminus Y$.
		
		\item \emph{Homotopy invariance}: if $\pi:E\to X$ is a finite rank vector bundle, then we have an isomorphism
		\[\pi^*: A^i(X,\M)\simeq A^i(E,\M)\]
		which preserves the cohomological degrees.
		
		\item \emph{Pullback along regular embeddings}: every regular embedding $i:X\to Y$ induces a pullback morphism
		\[ i^*:A^i(Y,\M)\longrightarrow A^i(X,\M) \]
		which satisfies the usual functorial properties.
		
		\item \emph{Pullback from smooth targets}: every morphism $f:X\to Y$ with smooth target $Y$ induces a pullback morphism
		\[ f^*:A^i(Y,\M)\longrightarrow A^i(X,\M) \]
		which satisfies the usual functorial properties. Whenever $f$ is flat of relative constant dimension, this pullback coincides with the flat pullback introduced before.
		
		\item \emph{Ring structure}: if $X$ is smooth and $\M$ is a cycle module with a pairing, then $A^{*}(X,\M)$ inherits the structure of a graded-commutative ring, where the graded-commutativity should be understood in the following sense: if $\alpha$ has codimension $i$ and degree $d$, and $\beta$ has codimension $j$ and degree $e$, then $\alpha\cdot\beta=(-1)^{de}\beta\cdot\alpha$, and the product has codimension $i+j$ and degree $d+e$.
		
		\item \emph{Module structure}: If $Y$ is smooth and we have pairings $\N\times \N\to \N$ and $\N\times \M\to \M$, then $A^{*}(Y,\M)$ inherits the structure of an $A^{*}(Y,\N)$-module.
		
		For every morphism $f:X\to Y$, every $\alpha$ in $A^{*}(Y,\N)$ and $\beta$ in $A^{*}(X,\M)$ we have:
		\[ f^*(\alpha\cdot\beta)=f^*\alpha\cdot f^*\beta \]
		
		For every proper morphism $f:X\to Y$,  the following projection formula holds:
		\[ f_*(\alpha\cdot f^*\beta)=f_*\alpha \cdot \beta \]
		
		\item \emph{Chern classes}: For a vector bundle $E\to X$ of finite rank $r$ we have well defined Chern class homomorphisms
		\[ c_k(E)(-):A^i(X,\M) \longrightarrow A^{i+k}(X,\M) \]
		which satisfy the usual standard properties.
		
		In particular, if $i:X\to E$ denotes the zero-section embedding, we have
		\[ i^*i_*(\alpha)=c_r(E)(\alpha) \]
		
		\item \emph{Projective bundle formula}: For a projective bundle $\pi:P(E)\to X$ whose fibres have dimension $r$, we have:
		\[ A^{i}(P(E),\M)\simeq \mathop{\oplus}_{j=0}^{i} c_1(\OO(1))^j\left( \pi^*A^{i-j}(X,\M) \right) \]
		for $0\leq i\leq r$. Moreover we have:
		\[ c_1(\OO(1))^{r+1}(-) = - \sum_{i=0}^{r} c_1(\OO(1))^i\left( \pi^*c_{r+1-i}(E)(-)\right)\]
	\end{enumerate}
\end{prop}
\begin{proof}[Sketch of proof]
	The proofs of (1)-(4) are the content of \cite{Rost}*{Sec. 4}, (5) is \cite{Rost}*{Prop. 8.6}, (6) is \cite{Rost}*{Prop. 12.3}, (7) is \cite{Rost}*{Thm. 12.1 and Prop. 12.2}, (8) and (9) are \cite{Rost}*{Thm. 14.6}. The theory behind (10) and (11) is developed in \cite{PirCohHypEven}*{Sec. 2.1}
\end{proof}
\begin{prop}\label{prop:chern}
Let $E$ be a vector bundle over a smooth variety $X$, and let $\M$ be a cycle module of $\ell$-torsion. Let $\alpha$ be an element of $A^{*}(X,\M)$. 

The class $c_i(E)(\alpha) \in A^{*}(X,\M)$ is equal to the image of $c_i(E)(1) \otimes \alpha$ through the multiplication map $A^{*}(X,\H_{\ZZ/\ell\ZZ}) \otimes A^{*}(X,\M) \rightarrow A^{*}(X,\M)$.
\end{prop}
\begin{proof}
Let $\pi:E\to X$ be a line bundle with zero section $i:X\to E$, so that $c_1(E)(\alpha)=i^*i_*\alpha$. 
We claim that 
\begin{equation}\label{eq:chern formula} 
i^*i_*(\alpha\cdot\beta)=i^*i_*\alpha\cdot \beta
\end{equation}
where $\alpha$ is an element of $A^{*}(X,\H_{\ZZ/\ell\ZZ})$ and $\beta$ is an element of $A^{*}(X,\M)$. Observe that $i^*=(\pi^*)^{-1}$: we will show that the pullback along $\pi$ of the two sides of (\ref{eq:chern formula}) coincide.
The pullback of the right hand hand side is:
\[ \pi^*(i^*i_*\alpha\cdot\beta)=\pi^*i^*i_*\alpha\cdot \pi^*\beta=i_*\alpha\cdot\pi^*\beta \]
The pullback of the left hand side is:
\[ \pi^*i^*i_*(\alpha\cdot\beta)=i_*(\alpha\cdot\beta)=i_*(\alpha\cdot i^*\pi^*\beta)=i_*\alpha\cdot\pi^*\beta \]
where in the last equality we have used the projection formula (see Proposition \ref{prop:properties}.(9)).

By taking $\alpha=1$ we obtain a proof of the Proposition in the case of line bundles. The general case can be deduced from this one by applying the splitting principle and the Whitney summation formula.
\end{proof}

\begin{lm}\label{lm:Gm torsor}
Let $X$ be a scheme and write $\Gm=\Spec(k\left[ t, t^{\operatorname{-1}}\right])$. Let $\M$ be an $\ell$-torsion cycle module. Then
\[A^{*}(X \times \Gm,\M) = A^{*}(X,\M) \oplus \lbrace t \rbrace A^{*}(X,\M),\]
where $\lbrace t \rbrace$ is seen as an element of $\H^{1}(k(t),\mu_\ell)\simeq k(t)^*/(k(t)^{*})^{\ell}$.
\end{lm}
\begin{proof}
Consider the localization long exact sequence induced by the open embedding $X\times\Gm\subset X\times\bA^1$, which is:
\[ \cdot\to A^i(X\times\bA^1,\M) \to A^i(X\times\Gm,\M)\to A^i(X\times\{0\},\M)\to A^{i+1}(X\times\bA^1, \M)\to\cdots \]
We have $A^i(X\times\bA^1,\M)\simeq A^i(X\times\{0\},\M)\simeq A^i(X,\M)$. With this identification, the last morphism in the sequence above coincides with multiplication by $c_1(\OO_X)=0$, hence for every $i$ we have:
\[ 0\to A^i(X,\M) \to A^i(X\times\Gm,\M) \to A^i(X,\M)\to 0  \]
Let $t$ be the element in $\Inv(X,\H_{\ZZ/\ell\ZZ})$ which sends a morphism $\Spec(F)\to X\times\Gm$ to the equivalence class in $\H^1(F,\mu_\ell)\simeq F^*/(F^*)^\ell$ of the element in $F^*$ defined by $\Spec(F)\to X\times\Gm \to \Gm$.

We can use the invariant $t$, regarded as an element of $A^0(X\times\Gm,\H_{\ZZ/\ell\ZZ})$ to define a morphism $i:A^i(X,\M)\to A^i(X\times\Gm,\M)$ by setting $i(\gamma)=t\cdot\pr_1^*\gamma$ (we are using here the fact that $A^{*}(X\times\Gm,\M)$ is an $A^{*}(X\times\Gm,\H_{\ZZ/\ell\ZZ})$-module).

It is easy to check that $i$ provides a splitting for the short exact sequence above, thus concluding the proof.
\end{proof}

\begin{prop}
Let $\M$ be an $\ell$-torsion cycle module. If $f:X\to Y$ is a universal homeomorphism, it induces an isomorphism $f_*: A^{*}(X,\M) \xrightarrow{\simeq} A^{*}(Y,\M)$.
\end{prop}
\begin{proof}
Let $f: X \rightarrow Y$ be a universal homeomorphism. Given a point $y \in Y$, its fibre $x$ is a point of $X$ and the map $f_x: x \rightarrow y $ is a purely inseparable field extension. 

The two formulas 
\[(f_x)_{*}( (f_x )^* \alpha ) = \left[ k(x) : k(y) \right] \alpha, \quad (f_x)^{*}( (f_x )_* \beta ) = \left[ k(x) : k(y) \right] \beta
\]

toghether with the fact that $\ell$ is not divisible by the characteristic of $k$, imply that both $(f_x)^{*}$ and $(f_x)_{*}$ are isomorphisms.
Then $f_*$ induces an isomorphism on cycle level, proving our claim.
\end{proof}

\section{The Brauer group of $\Mcal_{1,1}$}\label{sec:m11}

The tools we developed up to this point are enough for a first demonstration of how our methods work. 

The Brauer group of $\Mcal_{1,1}$ has been explored in depth by Antieau and Meier in \cites{AntMeiEll, Mei}, including in mixed characteristic, which is beyond the reach of our tools. 

Our proof has some independent interest though, as it is much simpler than the techniques used in their papers.

\begin{thm}\label{thm:Inv of M11}
Assume the characteristic of $k$ is different from $2$ or $3$, and let $\M$ be a $\ell$-torsion cycle module. Then the cohomological invariants of $\Mcal_{1,1}$ with coefficients in $\M$ are given by 
\[\Inv(\Mcal_{1,1},\M)\simeq\M^{\bullet}(k) \oplus \lbrace 27x^2 + 4y^3 \rbrace  \cdot \M^{\bullet}(k)_{12}.\]
\end{thm}
\begin{proof}
Consider the standard presentation of $\Mcal_{1,1}$ as the quotient $\left[ (\bA^2 \smallsetminus V) / \Gm \right]$, where $V$ is the curve $27x^2 + 4y^3 = 0$ and $\Gm$ acts by $(x, y) \mapsto (xt^6 , yt^4)$. The curve $V$ is universally homeomorphic to $\bA^1$, so we have an exact sequence
\[0 \rightarrow \M^{\bullet}(k) \rightarrow A^{0}(\bA^2 \smallsetminus V, \M) \rightarrow \M^{\bullet}(k) \rightarrow 0\]

The boundary of the element $\lbrace 27x^2 + 4y^3 \rbrace \in \H^{1}(k(x,y), \mu_\ell)$ at $V$ is $1$, and it is unramified everywhere else so it belongs to $A^{0}(\bA^2 \smallsetminus V, \H_{\ZZ/\ell\ZZ})$. In particular, the submodule $\lbrace 27x^2 + 4y^3 \rbrace  \cdot \M^{\bullet}(k)$ maps injectively to $\M^{\bullet}(k) = A^{0}(V,\M)$, splitting the exact sequence above. Thus 
\[A^{0}(\bA^2 \smallsetminus V, \M) = \M^{\bullet}(k)\oplus \lbrace 27x^2 + 4y^3 \rbrace  \cdot \M^{\bullet}(k).\]

We have to understand which of these elements glue to invariants of $\Mcal_{1,1}$. This is equivalent to check whether the two pullbacks through 
\[\pr_1, {\rm m}: (\bA^2 \setminus V) \times \Gm = (\bA^2 \setminus V) \times_{\Mcal_{1,1}} (\bA^2 \setminus V) \rightrightarrows \bA^2 \setminus V\]
coincide, where ${\rm m}$ denotes the multiplication map. We have 
\[A^{0}((\bA^2 \setminus V) \times \Gm, \M) = A^{0}(\bA^2 \setminus V, \M) \oplus t \cdot A^{0}(\bA^2 \setminus V, \M)\]
and ${\rm m}^*\lbrace 27x^2 + 4y^3 \rbrace = 12\lbrace t \rbrace + \lbrace 27x^2 + 4y^3 \rbrace$. Then an element $ \lbrace 27x^2 + 4y^3 \rbrace \cdot \alpha$ is unramified if and only if $\alpha \in \M^{\bullet}(k)_{12}$. 
\end{proof}

\begin{cor}
We have 
\[^c{\rm Br}(\Mcal_{1,1})\simeq {^c{\rm Br}(k)} \oplus {\rm H}^1(k, \ZZ/12\ZZ)\]
Moreover, every non-trivial element in the groups above is represented by a cyclic algebra.
\end{cor}
\begin{proof}
The formula for the cohomological Brauer group follows immediately by applying Theorem \ref{thm:Inv of M11} to $\M=\H_{\mu_{\ell}^{\vee}}$ for increasing $\ell$. Then Lemmas \ref{lm:H1} and \ref{lm:cyclic} allow us to conclude immediately.
\end{proof}

\section{Some equivariant computations}\label{sec:equiv}


We begin by briefly recalling the Edidin--Graham--Totaro equivariant approximation construction. After that, we will explicitly compute some equivariant Chow groups with coefficients: these computations will be frequently used in the remainder of the paper.

Consider an algebraic space $X$ being acted upon by an algebraic group $G$, and assume $G$ admits a faithful representation $G \subset {\rm GL}_n$. Then $G$ also admits a generically free representation $V$, and if we take the product $V_i= V \times \ldots \times V$ of $i$ copies of $V$ we know that $G$ acts freely on an open subset $U_i$ of $V_i$ whose complement has codimension at least $i$. As the action of $G$ on $X \times U_i$ is free, the quotient $\left[ X \times U_i /G\right]$ is an algebraic space.

Now, the map $\left[ X \times U_i /G\right] \rightarrow \left[ X/G\right]$ is the composition of a vector bundle and an open immersion whose complement has codimension at least $i$. Consider a graded theory of invariants $F^j$ such that:
\begin{itemize}
\item If $E \rightarrow Y$ is a vector bundle, then $F^j(Y)=F^j(E)$.
\item If $U \rightarrow Y$ is an open immersion whose complement has codimension $i$, then $F^j(Y)=F^j(U)$ for every $j< i$.
\end{itemize}
Then we have $F^j(\left[X/G\right])=F^j(\left[X\times U_i/G \right])$ for some sufficiently large $i$. 

An example of such a theory is \'etale cohomology when $X$ and $G$ are smooth, or singular cohomology when everything is defined over $\mathbb{C}$. 

We can also go the other way around: if the theory $F$ is only defined for algebraic spaces, and it has the properties above, then there is only one possible extension to quotient stacks that still satisfies the same properties, namely $F^j(\left[ X/G\right]) = F^j(\left[X\times U_i/G \right])$ for some large enough $i$. 

A simple double fibration argument shows that this does not depend on the choice of the representation of $G$, and moreover this can be proven to be independent of the presentation as well, so it's really an invariant of the stack $\left[X/G\right]$.

This construction, which already existed in the context of equivariant homology, was first used by Totaro \cite{Tot} to compute the Chow rings of some classifying stacks ${\rm B}G$, and was later extended by Edidin and Graham to general quotient stacks \cite{EG}. 

The construction carries over immediately to Chow groups with coefficients. Given an algebraic group $G$ acting a scheme $X$, we will denote the $G$-equivariant Chow groups with coefficients of $X$ by 
\[\bigoplus_{i\geq 0}A^{i}_G(X,{\rm M})=A^*_{G}(X,\M).\] This was first done by Guillot \cite{Guil}. Note that contrary to the case of schemes, there is no maximum bound for the codimension of a non-zero element.

These groups enjoy all the same properties as the ordinary Chow groups with coefficients. Sometimes we will refer to the equivariant cycle groups $C_{G}^{i}(X,{\rm M})$; by this we will always mean the cycles on an appropriate equivariant approximation.

Cohomological invariants turn out to be equal to the zero codimensional Chow group with coefficents for algebraic spaces, hence we can compute them on quotient stacks by using equivariant Chow groups with coefficients.

\begin{prop}
Let $X$ be a smooth quasi-separated algebraic space over $k$, being acted upon by a smooth affine group $G/k$. Then
\[\Inv(\left[X/G\right],\M) = A^0_G(X,\M).\]
\end{prop}
\begin{proof}
This is the content of \cite{PirAlgStack}*{Thm. 4.16} in the case where the coefficients are in étale cohomology. The general case can be proved in the exact same way, as an immediate consequence of \ref{cor:hom inv}.
\end{proof}

In the rest of Section we will perform some explicit computations of equivariant Chow rings with coefficients. These results will be frequently used in the remainder of the paper.

\begin{prop}\label{prop:AGm}
We have
\[ A^{*}_{GL_m}(\Spec(k), \M)={\rm CH}^{*}_{GL_m}(\Spec(k)) \otimes {\rm M}^{\bullet}(k).\]
\end{prop}
\begin{proof}
Let $V$ be a finite dimensional vector space and let ${\rm Gr}_m(V)$ be the grassmannian of $m$-subspaces of $V$. We are going to prove that:
\begin{equation*}\label{eq:AGr} {\rm CH}^{i}({\rm Gr}_m(V))\otimes {\rm M}^{\bullet}(k) \simeq A^{i}({\rm Gr}_m(V),{\rm M}) \end{equation*}
for every $i\geq 0$ and $0<m\leq{\rm dim}(V)$. This will readily imply the Proposition, because of the isomorphisms:
\[ {\rm CH}^i({\rm Gr}_m(V))\simeq {\rm CH}^i_{GL_m}(\Spec(k)),\quad A^i({\rm Gr}_m(V),{\rm M})\simeq A^i_{GL_m}(\Spec(k),{\rm M}) \]
for vector spaces $V$ of sufficiently high dimension (see \cite{EG}*{Subsec. 3.2}).

Let ${\rm Fl}_m(V)$ be the scheme of complete flags of length $m$ in $V$. Observe that ${\rm Fl}_m(V)$ can be constructed as a tower of projective bundles both over $\Spec(k)$ and ${\rm Gr}_m(V)$.

If the Chow ring of a smooth scheme $X$ is generated by Chern classes of vector bundles, thanks to Proposition \ref{prop:chern} there is a well defined multiplication morphism:
\[ {\rm CH}^i(X)\otimes {\rm M}^{\bullet}(k) \rightarrow A^i(X,{\rm M}) \]
Moreover, if $f:Y\to X$ is a flat morphism of relative constant dimension from a smooth scheme whose Chow ring is also generated by Chern classes, it is easy to check that the pullback of a product is equal to the product of pullbacks.

Applying this to the morphism ${\rm Fl}_m(V)\to {\rm Gr}_m(V)$ we obtain the following commutative diagram:
\begin{equation}\label{eq:diag}
\xymatrix{ {\rm CH}^i({\rm Gr}_m(V))\otimes {\rm M}^{\bullet}(k) \ar[r] \ar[d] & {\rm CH}^i ({\rm Fl}_m(V)) \otimes {\rm M}^{\bullet}(k) \ar[d] \\
   A^i({\rm Gr}_m(V),{\rm M}) \ar[r] & A^i({\rm Fl}_m(V),\M)  } 
\end{equation}
The fact that ${\rm Fl}_m(V)$ is a tower of projective bundles over a point combined with an iterated application of \cite{PirCohHypEven}*{Prop. 2.4} shows that the right vertical morphism in the diagram above is an isomorphism.

Define a splitting $s$ of the bottom horizontal arrow of (\ref{eq:diag}) as follows: at each level of the tower of projective bundles ${\rm Fl}_m(V)\to {\rm Gr}_m(V)$, multiply $r$ times with the appropriate hyperplane section (here $r$ is the dimension of the fibre at that level) and then take the pushforward to the level below. In the same way we can construct a splitting $s'$ at the level of Chow groups.

The pullback $f^*:{\rm CH}^{*}({\rm Gr}_m(V))\to {\rm CH}^{*}({\rm Fl}_m(V))$ is then split injective and from this it is easy to deduce that the top horizontal arrow in the diagram (\ref{eq:diag}) is injective, hence the left vertical arrow must be injective as well.

To show surjectivity, observe that if $\alpha$ is an element in $A^{*}({\rm Gr}_m(V),{\rm M})$, we have
\[ \alpha=s(f^*(\alpha))=s(\sum \xi_i \cdot \beta_i)=\sum s'(\xi_i) \cdot \beta_i \]
for some $\xi_i$ in ${\rm CH}^{*}({\rm Fl}_m(V))$ and $\beta_i$ in ${\rm M}^{\bullet}(k)$. Therefore the left vertical morphism of (\ref{eq:diag}) is surjective and this concludes the proof.
\end{proof}
\begin{prop}\label{prop:Amun}
We have
\[A^{*}_{\mu_\ell}(\Spec(k),{\rm M})=(\ZZ\left[s\right]/\ell s) \otimes {\rm M}^{\bullet}(k) \mathop{\oplus}_n s^n \alpha\cdot  \M^{\bullet}(k)_{\ell} ,\]
where $s$ has codimension $1$, cohomological degree $0$ and $\alpha$ has codimension $0$ and cohomological degree $1$.
In particular, if $M$ is of $\ell$-torsion
\[A^{*}_{\mu_\ell}(\Spec(k),\M) \simeq A^{*}_{\mu_\ell}(\Spec(k),{\rm H}_{\ZZ/\ell\ZZ}) \otimes \M^{\bullet}(k).\]
\end{prop}

\begin{proof}
We have ${\rm B}\mu_\ell=[(\bA^1\smallsetminus\{0\})/\Gm]$, where the action is defined as $\lambda\cdot x=\lambda^\ell x$. 

The localization long exact sequence for the $\Gm$-equivariant embedding $\{0\}\hookrightarrow\bA^1$ reads as follows:
\[ \cdots \to A^i_{\Gm}(\bA^1,\M)\to A^i_{\Gm}(\bA^1\smallsetminus\{0\},\M) \to A^i_{\Gm}(\{0\},\M) \to A^{i+1}_{\Gm}(\bA^1,\M) \to \cdots \]
By Proposition \ref{prop:AGm} we have:
\[ A_{\Gm}^i(\bA^1,\M)\simeq A_{\Gm}^i(\{0\},\M) \simeq \M^{\bullet}(k) s^i .\]
Moreover by Proposition \ref{prop:chern} the pushforward $i_*:A^i_{\Gm}(\{0\},\M) \to A^{i+1}_{\Gm}(\bA^1,\M)$ corresponds to multiplication by $[\{0\}]=\ell s$.

Proposition \ref{prop:AGm} assures us that multiplying by $s$ defines an injective morphism, hence we deduce:
\[{\rm ker}(i_*)=\M^{\bullet}(k)_{\ell}\cdot s^i,\quad {\rm im}(i_*)=\M^{\bullet}(k)\cdot \ell s^{i+1} .\]
It follows that for every $i\geq 0$ we have the following short exact sequence:
\[ 0\to {\rm CH}^i_{\mu_\ell}(\Spec(k))\otimes\M^{\bullet}(k) \to A_{\Gm}^i(\bA^1\smallsetminus\{0\},\M) \to \M^{\bullet}(k)_{\ell}\cdot s^i \to 0 .\]
Recall from Lemma \ref{lm:Gm torsor} that there is a splitting 
\[A^i(\{0\},\M)\rightarrow A^i(\bA^1\smallsetminus\{0\},\M)\]
given by multiplication by the element $\{t\} \in A^0(\bA^1\smallsetminus\{0\},\K_{\M}^1)$. 

If we consider $\{t\} \cdot \tau$, where $\tau \in \M(k)$ is an element of $\ell$-torsion, this element is invariant for the $\Gm$-action and thus it glues to an element $\gamma \in A^0_{\Gm}(\bA^1\smallsetminus\{0\}, \M)$, splitting the exact sequence for $i=0$. 

Combining this splitting with multiplication by $s^i$ we obtain a splitting for every codimension $i$, proving our claim. 

\end{proof}

\begin{cor}\label{cor:Amun2}
Let $\M$ be a torsion cycle module, and let $\mu_\ell$ act trivially on a scheme $X$. Then
\[A^{*}_{\mu_\ell}(X,\M) \simeq (\ZZ\left[s\right]/\ell s)\otimes A^{*}(X,\M) \oplus_{n\geq 0} s^n\alpha A^{*}(X,\M)_\ell .\]

In particular, 
\[A^{0}_{\mu_\ell}(X,\M)=A^{0}(X,\M)\oplus \alpha A^{0}(X,\M)_\ell.\]
\end{cor}
\begin{proof}
Note that if $\Gm$ acts trivially on $X$ then $A^i_{\Gm}(X,\M)=A^i(X\times \mathbb{P}^r_k,\M)$ for some $r>i$, showing that \[A^{*}_{\Gm}(X,\M)=A^{*}(X,\M)\left[s\right]= {\rm CH}^{*}_{\Gm}(\Spec(k))\otimes A^{*}(X,\M).\]

Using this we can just repeat the proof of the previous Lemma using the $\ell$-twisted action of $\Gm$ on $X \times \Gm$ and the compatibility of the pullback \[A^{*}_{\Gm}(X,\M)\rightarrow A^{*}_{\Gm}(X\times \Gm,\M)\] with the long exact sequence.
\end{proof}

\begin{prop}\label{prop:ABPGL2 odd}
Let $\M$ be an odd torsion cycle module. Then
\[A^{*}_{{\rm PGL}_2}(\Spec(k),\M)={\rm CH}^{*}_{{\rm PGL}_2}(\Spec(k))\otimes \M^{\bullet}(k).\]
\end{prop}

The strategy of proof is the following: we first compute $A^{*}({\rm O}_n(\Spec(k)),\M)$ for $n=1,2,3$ and after that we exploit the isomorphism ${\rm PGL}_2\simeq {\rm SO}_3$ to conclude our computation. 

The proof is borrowed almost verbatim from \cite{PirCohHypThree}*{Cor. 1.10}, albeit we have to deal here with a generic torsion cycle module $\M$, which in particular may not possess a pairing.

\begin{lm} \label{lm:AOn}
    Let $\M$ be an odd torsion cycle module. Then we have:
    \begin{itemize}
        \item $A^*_{O_1}(\Spec(k),\M)= \M^{\bullet}(k)$.
        \item $A^*_{O_2}(\Spec(k),\M)\simeq\M^{\bullet}(k)[c_2]\simeq {\rm CH}_{O_2}^*(\Spec(k))\otimes\M^{\bullet}(k)$.
        \item $A^{*}_{O_3}(\Spec(k),\M)\simeq \M^{\bullet}(k)[c_2]\simeq {\rm CH}^{*}_{O_3}(\Spec(k))\otimes\M^{\bullet}(k)$.
    \end{itemize}
\end{lm}
\begin{proof}
Let $\M$ be an odd torsion cycle module. Using the description of $A^{*}_{\mu_2}(\Spec(k),\M)$ given by Proposition \ref{prop:Amun}, we see that every element of codimension $>0$ is both of $2$-torsion and $\ell$-torsion, hence $0$. We deduce:
\begin{equation}\label{eq:Amu2}
A^{*}_{O_1}(\Spec(k),\M)\simeq A^{*}_{\mu_2}(\Spec(k),\M)\simeq \M^{\bullet}(k)
\end{equation}

Let $V$ be an $n$-dimensional vector space endowed with a non-degenerate quadratic form $q$. Set ${\rm O}_n=O(V,q)$ and define $C\subset V\smallsetminus\{0\}$ as the vanishing locus of $q$. Let $B$ be the complement of $C$ in $V\smallsetminus\{0\}$. All these schemes are ${\rm O}_n$-invariant.

We can compute $A_{{\rm O}_n}^{*}(B,\M)$ as follows: consider the morphism $B\to\Gm$ induced by $q$ and form the cartesian square
\[ \xymatrix{ \widetilde{B} \ar[r] \ar[d] & B \ar[d] \\
\Gm \ar[r] & \Gm } \]
where the bottom morphism is the square map. We have $\widetilde{B}\simeq Q\times\Gm$, where $Q\subset B$ is the fibre over the unit in $\Gm$, and moreover both horizontal arrows are $\mu_2$-torsors. Therefore:
\begin{equation}\label{eq:AB first}
A_{{\rm O}_n\times\mu_2}^{*}(Q\times\Gm,\M)\simeq A_{{\rm O}_n}^{*}(B,\M) 
\end{equation}
Let $\Lcal$ be the $({\rm O}_n\times\mu_2)$--equivariant line bundle over $Q$ such that the equivariant $\Gm$-torsor $Q\times\Gm\to Q$ is the complement in $\Lcal$ of the zero section. The associated localization exact sequence is:
\begin{equation} \label{eq:seq for Q}
\begin{tikzcd}
  \cdots \rar &  A^i_{{\rm O}_n\times\mu_2}(Q\times\Gm,\M) \rar
             \ar[draw=none]{d}[name=X, anchor=center]{}
    & A^i_{{\rm O}_n\times\mu_2}(\Lcal,\M) \ar[rounded corners,
            to path={ -- ([xshift=2ex]\tikztostart.east)
                      |- (X.center) \tikztonodes
                      -| ([xshift=-2ex]\tikztotarget.west)
                      -- (\tikztotarget)}]{dll}[at end]{} \\      
   A^i_{{\rm O}_n\times\mu_2}(Q,\M) \rar & A^{i+1}_{{\rm O}_n\times\mu_2}(Q\times\Gm,\M) \rar & \cdots 
\end{tikzcd}
\end{equation}
Via the identification $A^{*}_{{\rm O}_n\times\mu_2}(\Lcal,\M)\simeq A^{*}_{{\rm O}_n\times\mu_2}(Q,\M)$, we see that the last map is given by the multiplication by $c_1^{{\rm O}_n\times\mu_2}(\Lcal)$.

The action of ${\rm O}_n\times\mu_2$ on $Q$ is transitive: its stabilizer is ${\rm O}_{n-1}\times\mu_2$, hence $[Q/{\rm O}_n]\simeq B({\rm O}_{n-1}\times\mu_2)$ and
\begin{equation} \label{eq:AQ}
A_{O_{n}\times\mu_2}^{*}(Q,\M)\simeq A_{{\rm O}_{n-1}\times\mu_2}^{*}(\Spec(k),\M) \simeq A_{{\rm O}_{n-1}}^{*}(\Spec(k),\M)
 \end{equation}
where the last isomorphism is a consequence of the triviality of $A_{\mu_2}^{*}(\Spec(k),\M)$.

By Proposition \ref{prop:chern}, we have \[c_1^{{\rm O}_n\times\mu_2}(\Lcal)(\alpha)=c_1^{{\rm O}_n\times\mu_2}(\Lcal)(1)\cdot\alpha\] for every $\alpha$ in $A^{*}_{{\rm O}_n\times\mu_2}(Q,\M)$, where $c_1^{{\rm O}_n\times\mu_2}(\Lcal)$ is regarded as an element inside $A^1_{{\rm O}_n\times\mu_2}(Q,{\rm H}_{\ZZ/\ell\ZZ})$.

Observe that from (\ref{eq:AQ}) it follows that the part of $A^1_{{\rm O}_n\times\mu_2}(Q,{\rm H}_{\ZZ/\ell\ZZ})$ of cohomological degree $0$ is equal to $\ZZ/\ell\ZZ \cdot c_1/\langle 2\cdot c_1 \rangle=0$. Therefore, from the sequence (\ref{eq:seq for Q}) we deduce the exact sequence:
\begin{equation*}
0\to A^i_{O_{n}\times\mu_2} (Q,\M) \to  A^i_{O_{n}\times\mu_2} (Q\times\Gm,\M) \to A^i_{{\rm O}_n\times\mu_2}(Q,\M)\to 0
\end{equation*}
Combining this with the content of (\ref{eq:AQ}) and (\ref{eq:AB first}), we get:
\begin{equation}\label{eq:AB}
A^{*}_{{\rm O}_n}(B,\M)\simeq A^{*}_{{\rm O}_{n-1}}(\Spec(k),\M)\oplus A^{*}_{{\rm O}_{n-1}}(\Spec(k),\M)\cdot\beta 
\end{equation}
as $\Het_{\ZZ/\ell\ZZ}(k)$-modules, where $\beta$ has codimension $0$ and cohomological degree $1$.

Next we compute $A_{{\rm O}_n}^{*}(C,\M)$, where $C$ is the vanishing locus of $q$ in $V\smallsetminus\{0\}$. The action of ${\rm O}_n$ on $C$ is transitive: the stabilizer is given by semidirect product of $O_{n-2}$ with a subgroup $H$ of the group of affine transformations of a vector space $W$. More precisely, the group $H$ is isomorphic to $W$ via the natural projection of ${\rm Aff}(V)$ on translations. It follows from \cite[Lemma 1.7]{PirCohHypThree} that:
\begin{equation}\label{eq:AC}
A_{{\rm O}_n}^{*}(C,\M)\simeq A_{O_{n-2}\ltimes H}^*(\Spec(k),\M)\simeq A_{O_{n-2}}^*(\Spec(k),\M)
\end{equation}

Next we compute $A^{*}_{O_2}(\Spec(k),\M)$. First observe that $O_0=\{id\}$ and $O_1\simeq\mu_2$. Consider the pushforward morphism
\begin{equation*}
i_*:A^{*}_{O_2}(C,\M)\simeq \M^{\bullet}(k) \rightarrow A^{*+1}_{O_2}(V\smallsetminus\{0\},\M) 
\end{equation*}
The isomorphism on the left side comes from (\ref{eq:AC}). By the projection formula $i_*\alpha=\alpha\cdot i_*(1)$, and $i_*(1)$ is in degree $0$ and codimension $1$. We have $A^1_{O_2}(V\smallsetminus\{0\},\M)\simeq A^1_{O_2}(V,\M)$ which is in turn isomorphic to $A^1_{O_2}(\Spec(k),\M)$. 

The piece of cohomological degree $0$ of the latter group turns out to be equal to ${\rm CH}^1_{O_2}(\Spec(k))\otimes M^0(k)$, which vanishes because it is of both $2$-torsion and $\ell$-torsion. We deduce that $i_*=0$.

This last remark, applied to the localization exact sequence for $C\subset V\smallsetminus\{0\}$, implies that the following short sequence is exact:
\begin{equation}\label{eq:sesAO2Vminus}
0\to A^{*}_{O_2}(V\smallsetminus\{0\},\M)\to A^{*}_{O_2}(B,\M) \to A^{*}_{O_2}(C,\M) \to 0
\end{equation}
By (\ref{eq:AC}) the group on the right is equal to $\M^{\bullet}(k)$ in codimension $0$ and it is $0$ in higher codimension. By (\ref{eq:AB}), the group in the middle is equal to $A^{*}_{\mu_2}(\Spec(k),\M)\oplus A^{*}_{\mu_2}(\Spec(k),\M)\cdot\beta$ in codimension $0$ and it vanishes in higher codimension, where $\beta$ has cohomological degree $1$.

From (\ref{eq:Amu2}) we know that $A^{*}_{\mu_2}(\Spec(k),\M)$ is trivial, hence the short exact sequence (\ref{eq:sesAO2Vminus}) implies:
\begin{equation}\label{eq:AO2Vminus}
A^{*}_{O_2}(V\smallsetminus\{0\},\M)\simeq \M^{\bullet}(k) 
\end{equation}
Now consider the localization exact sequence
\begin{equation}\label{eq:sesforAO2V}
\cdots\to A^i_{O_2}(V,\M) \to A^i_{O_2}(V\smallsetminus\{0\},\M) \to A^{i-1}_{O_2}(\{0\},\M) \to A^{i+1}_{O_2}(V,\M)\to\cdots
\end{equation}
induced by the open embedding $V\smallsetminus \{0\}\hookrightarrow V$.

Observe that $A^0_{O_2}(V,\M)\simeq A^0_{O_2}(V\smallsetminus\{0\})$ because the codimension of $\{0\}$ in $V$ is $>1$. Moreover, from (\ref{eq:AO2Vminus}) we know that $A^i_{O_2}(V\smallsetminus\{0\},\M)=0$ for $i>0$. Observe also that the pushforward induced by the closed embedding $\{0\}\hookrightarrow V$ coincides with multiplication by $c_2^{O_2}(V)=:c_2$.

Henceforth from the long exact sequence (\ref{eq:sesforAO2V}) we readily deduce:
\begin{itemize}
	\item $A^0_{O_2}(V,\M)\simeq \M^{\bullet}(k)$.
	\item $A^i_{O_2}(V,\M)=0$ for $i$ odd.
	\item $A^{i+2}_{O_2}(V,\M)\simeq c_2\cdot A^i_{O_2}(V,\M)$
\end{itemize}
Therefore:
\begin{equation}\label{eq:AO2}
A^{*}_{O_2}(\Spec(k),\M)\simeq\M^{\bullet}(k)[c_2]\simeq {\rm CH}_{O_2}^*(\Spec(k))\otimes\M^{\bullet}(k)
\end{equation}
Next we prove that the same result holds for $O_3$. As before, we start by computing $A_{O_3}^{*}(V\smallsetminus\{0\},\M)$. 

The triviality of $A^{*}_{\mu_2}(\Spec(k),\M)$ together with (\ref{eq:AC}) implies that $A_{O_3}^{*}(C,\M)\simeq\M^{\bullet}(k)$. The same argument used for $O_2$ shows that the pushforward morphism $i_*:A^{*}_{O_3}(C,\M)\to A^{*}_{O_3}(V\smallsetminus\{0\},\M)$ is zero.

Plugging this information into the localization exact sequence induced by the open embedding $B\hookrightarrow V\smallsetminus\{0\}$, we deduce that $A^i_{O_3}(V\smallsetminus\{0\},\M)\simeq A^i_{O_3}(B,\M)$ for $i>0$ and that the following is exact:
\begin{equation*}
0\to A^0_{O_3}(V\smallsetminus\{0\},\M) \to A^0_{O_3}(B,\M) \to A^0_{O_3}(C,\M)\simeq \M^{\bullet}(k) \to 0
\end{equation*}
By (\ref{eq:AB}) together with (\ref{eq:AO2}) we have:
\[ A_{O_3}^{*}(B,\M)\simeq \M^{\bullet}(k)[c_2] \oplus \M^{\bullet}(k)[c_2] \cdot \beta \]
with $\beta$ of codimension $0$ and cohomological degree $1$. We deduce that
\begin{equation*}
A^{*}_{O_3}(V\smallsetminus\{0\},\M)\simeq \M^{\bullet}(k)[c_2]\oplus\M^{\bullet}(k)[c_2]\cdot\gamma 
\end{equation*}
where $\gamma$ has codimension $2$ and cohomological degree $1$.

Observe that the pushforward morphism $A^{*}_{O_3}(\{0\},\M)\to A^{*+3}_{O_3}(V,\M)$ coincides with multiplication by $c_3(1)$, which is both a $2$-torsion and an $\ell$-torsion element, hence zero. 

Consider now the open embedding $V\smallsetminus\{0\}\hookrightarrow V$. By what we have just proved, the following short sequences are exact for every $i$:
\begin{equation*}
0\to A^i_{O_3}(V,\M) \to A^i_{O_3}(V\smallsetminus\{0\},\M)\to A^{i-2}_{O_3}(\{0\},\M)\to 0
\end{equation*}
We immediately deduce that:
\begin{itemize}
	\item $A^0_{O_3}(V,\M)\simeq \M^{\bullet}(k)$.
	\item $A^i_{O_3}(V,\M)=0$ for $i$ odd.
	\item $A^{2i}_{O_3}(V,\M)\simeq\M^{\bullet}(k)\cdot c_2^i$.
\end{itemize}
Therefore:
\begin{equation}\label{eq:AO3}
A^{*}_{O_3}(\Spec(k),\M)\simeq \M^{\bullet}(k)[c_2]\simeq {\rm CH}^{*}_{O_3}(\Spec(k))\otimes\M^{\bullet}(k)
\end{equation}
which concludes the proof of the Lemma.
\end{proof}

\begin{proof}[Proof of Prop. \ref{prop:ABPGL2 odd}]
Observe that ${\rm PGL}_2\simeq {\rm SO}_3$ and $O_3\simeq {\rm SO}_3\times\mu_2$. The result is then a direct consequence of Corollary \ref{cor:Amun2} and Lemma \ref{lm:AOn}.
\end{proof}

\section{The generalized cohomological invariants of ${\rm S}_n$ and ${\rm PGL}_2$}\label{sec: Sn PGL2}


In this Section we extend some results on classical cohomological invariants from \cite{GMS} to work with coefficients in a general module. We will apply Gille and Hirsch's splitting formula for generalized cohomological invariants \cite{GilHir}*{Thm. 4.8}.

We begin by discussing the relations between the cycle modules ${\rm H}_{\ZZ/\ell\ZZ}=\K_{\ell}$ as $\ell$ varies. For the remainder of this section, by $\eta$ we will always mean a positive integer not divisible by ${\rm char}(k)$. Consider the exact sequence

\[0 \rightarrow \ZZ/\eta\ZZ \rightarrow \ZZ/\ell \eta\ZZ \rightarrow \ZZ/\ell\ZZ \rightarrow 0 \]
We can twist it to get exact sequences 
\[0 \rightarrow \ZZ/\eta\ZZ (i) \rightarrow \ZZ/\ell \eta\ZZ (i) \rightarrow \ZZ/\ell\ZZ (i) \rightarrow 0 \]
In particular, when $i=1$ we retrieve the exact sequence 
\[0\rightarrow \mu_{\eta} \rightarrow \mu_{\ell \eta} \rightarrow \mu_{\ell} \rightarrow 0. \]
Consequently, for each $i$ we have exact sequences
\[ {\rm H}^i(-,\ZZ/\eta\ZZ (i))\rightarrow {\rm H}^i(-,\ZZ/\ell \eta\ZZ (i)) \rightarrow {\rm H}^i(-, \ZZ/\ell\ZZ (i)) \]
these maps form an exact sequence of cycle modules
\[ {\rm H}_{\ZZ/\eta\ZZ} \rightarrow {\rm H}_{\ZZ/\ell \eta\ZZ} \rightarrow {\rm H}_{\ZZ/\ell \ZZ}. \]
Using Voevodsky's norm-residue isomorphism \cite{Voe}*{Thm. 6.1}, we conclude that this is just the exact sequence 
\[ K_{\rm \eta} \rightarrow K_{\rm \eta \ell} \rightarrow K_{\ell} \rightarrow 0\]
In particular, the last map is surjective, and the kernel is just the image of the kernel of $\K_{\rm M} \rightarrow \K_{\ell}$, that is, the ideal $(\ell)$.

\begin{lm}\label{InvImage}
Let $X/k$ be a smooth scheme, and let $\alpha \in \K_{\ell}$. Let $\alpha'$ be an inverse image of $\alpha$ in $\K_{\rm M}^{\bullet}(k(X))$ (or in $\K_{\ell \eta}^{\bullet}(k(X))$. If $\alpha$ is unramified then $\alpha ' \cdot x$ is unramified for any $x \in A^0(X,\K_{\rm M})_{\ell}$ (resp. $A^0(X,\K_{\ell \eta})_{\ell}$). Moreover, the class $\alpha'\cdot x$ does not depend on the choice of $\alpha'$.
\end{lm}
\begin{proof}
By the compatibility of the morphism of cycle modules with the boundary map, the ramification of $\alpha'$ at any point $x$ of codimension $1$ has to belong to the kernel of $\K_{\rm M}^{\bullet}(k(x)) \rightarrow \K_{\ell }^{\bullet}(k(x))$, which means it is a multiple of $\ell$, so $\alpha' \cdot x$ is unramified as $\partial_v(\alpha' \cdot x)=\partial_v(\alpha')\cdot x = 0$. By the same reasoning, note that if $\alpha''$ is another inverse image of $\alpha$, the difference $\alpha' - \alpha''$ has to be a multiple of $\ell$, concluding the proof. 
\end{proof}

Using this, we can define a multiplicative action of the invariants with coefficients in $\K_{\ell}$ on the $\ell$-torsion of the invariants with coefficients in any cycle module.

\begin{lm}\label{lm:multiplication}
Let $X/k$ be a smooth scheme endowed with a $G$-action. There is a well defined multiplication:
\[ A^{0}_G(X,\K_{\ell})\otimes A^0_G(X,{\rm M})_\ell \rightarrow A^0_G(X,{\rm M}) \]
which for any $G$-equivariant morphism $f:X\to Y$ of smooth $G$-schemes satisfies the formula $f^*(\alpha\cdot\beta)=f^*\alpha\cdot f^*\beta$.
\end{lm}
\begin{proof}
For the sake of simplicity, we give a proof in the non-equivariant setting. The general case can be deduced in the same way using equivariant approximation.

Let $\alpha$ be an element of $A^0(X,\K_{\ell})\subset \K_{\ell}^{\bullet}(k(X))$ and let $\beta$ be an element of $A^0_G(X,{\rm M})_\ell$. For every inverse image $\alpha'$ of $\alpha$ in $\K_{\rm M}^{\bullet}(k(X))$, we have a well defined product $\alpha'\cdot\beta$ in ${\rm M}^{\bullet}(k(X))$: this product is unramified on $X$ because the ramification of $\alpha'$ at any point of codimension $1$ is a multiple of $\ell$, just as in the proof of the Lemma above.

By the same reasoning, the definition of the product does not depend on the choice of an inverse image $\alpha'$. The compatibility with the pullback follows from the compatibility of the product $A^0(X,\K_{\rm M})\otimes A^0(X,{\rm M}) \to A^0(X,{\rm M})$.
\end{proof}

This allows us to easily describe the cohomological invariants of $\mu_\ell^r$:

\begin{lm}\label{lm:Inv mu2r}
The cohomological invariants of ${\rm B}(\mu_\ell^r)$ with coefficients in ${\rm M}$ are given by
\[\Inv(\Brm (\mu_\ell^r),{\rm M}) \simeq {\rm M}^{\bullet}(k) \oplus\left(\mathop{\oplus}_{J \subseteq \left[ r\right] }\alpha_J \cdot {\rm M}^{\bullet}(k)_\ell \right) \]
Where $\alpha_{\lbrace j \rbrace}$ is the pullback of the identity invariant through the projection \[{\rm Pr}_{j}:\mu_\ell^r \rightarrow \mu_\ell\] and if $J= \lbrace j_1, \ldots, j_m \rbrace$ then $\alpha_{J} = \alpha_{\lbrace j_1 \rbrace} \ldots \alpha_{\lbrace j_m \rbrace}$.
\end{lm}

In particular, if ${\rm M}$ has no $\ell$-torsion, the invariants of $\mu_\ell^r$ with coefficients in ${\rm M}$ are all trivial.
\begin{proof}
This is an immediate consequence of Proposition \ref{prop:Amun} and Corollary \ref{cor:Amun2}. 
\end{proof}

Now, let $U_n \subset {\rm GL}_n$ be the subscheme of symmetric matrices. Consider the commutative diagram:

\[ \xymatrix{
			\Gm^n \ar[r]^{i} \ar[d] & U_n \ar[d] \\
			{\rm B}\mu_2^n \ar[r] & {\rm BO}_n } \]
			
The vertical maps are given respectively by the quotient by $\Gm^n$ acting on itself with weight two and the quotient by ${\rm GL}_n$ acting by $(A, S) \rightarrow A^{\rm T} S A$. In particular we can see the action of $\Gm^n$ on itself as the subgroup of diagonal matrices of ${\rm GL}_n$ acting on $\Gm^n \subset U_n$. The bottom map comes from the inclusion of the diagonal matrices with coefficients $\pm 1$ into ${\rm O}_n$. Note that both vertical maps are quotients by special groups and thus smooth-Nisnevich.

It's a well known fact that in characteristic different from two every symmetric matrix is equivalent to a diagonal matrix under the action of ${\rm GL}_n$. An immediate consequence of this fact is that the map from $\Gm^n$, and consequently the map from ${\rm B}\mu_2^n$, to ${\rm BO}_n$ are smooth-Nisnevich.

\begin{prop}\label{prop:Inv On}
Any nontrivial cohomological invariant of ${\rm O}_n$ is of $2$-torsion. We have
\[ \Inv({\rm BO}_n,{\rm M})={\rm M}^{\bullet}(k) \oplus \alpha_1 \cdot {\rm M}^{\bullet}(k)_2 \oplus \ldots \oplus \alpha_n \cdot {\rm M}^{\bullet}(k)_2\]
where $\alpha_i$ is the $i$-th symmetric function in $\alpha_{\lbrace 1 \rbrace}, \ldots, \alpha_{\lbrace n \rbrace} \in \Inv(\mu_2^n,\K_{2})$.
\end{prop}
\begin{proof}
The fact that all nontrivial invariants are of $2$-torsion is an obvious consequence of the inclusion $\Inv(\Brm {\rm O}_n,\M) \subset \Inv(\Brm \mu_2^n,\M)$. Note that moreover we have $\Inv({\rm BO}_n,{\rm M}) \subset \Inv(\Brm\mu_2^n, {\rm M})^{{\rm S}_n}$ by the gluing conditions, and $\Inv(\Brm\mu_2^n, {\rm M})^{{\rm S}_n}$ is exactly the group in our formula, so all we have to do is prove surjectivity. 

The description of the cohomological invariants with coefficients in a $\K_2$ is obtained in \cite[Sec. 23]{GMS}; the general description will follow by considering the product map $\Inv({\rm BO}_n,\K_2)_2 \otimes {\rm M}(k)_2 \to \Inv({\rm BO}_n, {\rm M})$. We have a commutative square

\[ \xymatrix{
			\Inv({\rm BO}_n,\K_2) \otimes {\rm M}(k)_2, \ar[r]^{\cdot} \ar[d] & \Inv({\rm BO}_n, {\rm M})_2 \ar[d] \\
		 \Inv(\Brm\mu_2^n, \K_2)^{{\rm S}_n} \otimes {\rm M}(k)_2 \ar[r]^{\cdot} &  \Inv(\Brm\mu_2^n, {\rm M})^{{\rm S}_n}_2 }
		\]
The bottom horizontal map and the vertical map on the left are surjective, proving our claim.
\end{proof}

\begin{rmk}
In \cite{Guil}*{Sec. 6}, Guillot claims that there are no Stiefiel-Whitney classes with coefficients in $\K_{\rm M}$, but he is only computing the degree $1$ component. An example of such a class in degree $2$ is 
\[\alpha_1 \lbrace -1 \rbrace \in {\rm Inv}^2({\rm BO}_n, \K_{\rm M}) 
\]
when $k$ does not contain a square root of $-1$.
\end{rmk}

\begin{thm}[Gille, Hirsch]\label{thm:split}
Let $(V,b) $ be a finite dimensional, regular symmetric bilinear space over $k$ and let $W \subset {\rm O}(V,b)$ be a finite subgroup such that $S(V)^{W}$ is a polynomial ring over $k$. Let $G_1,\ldots, G_r$ be different maximal abelian subgroups generated by reflections, representing any such subgroup up to conjugation. 
Then the pullback
\[\Inv({\rm B}(W),{\rm M}) \to \prod_{i} \Inv({\rm B}(G_i),{\rm M})^{N_{W}(G_i)} \]
is injective.
\end{thm}
\begin{proof}
This is a combination of \cite{GilHir}*{Thm. 4.8, Cor. 4.12}.
\end{proof}

\begin{rmk}
Gille and Hirsch also prove the same statement when $W$ is a Weyl group and the characteristic of $k$ is not a torsion prime for its root system. Moreover, they also prove the splitting principle for invariants with coefficients in Witt groups, which do not form a cycle module but have strong enough properties that the (classical) theory of cohomological invariants still works.
\end{rmk}

\begin{cor}
Let $W$ be as above. Than any non-constant cohomological invariant of ${\rm B}W$ is of $2$-torsion, i.e. \[2\cdot \left( \Inv({\rm B}(W), {\rm M})/{\rm M}^{\bullet}(k)\right)=0.\]
\end{cor}
\begin{proof}
This is immediate from Lemma \ref{lm:Inv mu2r}.
\end{proof}

\begin{prop}\label{prop:Inv Sn and PGL2}
We have:
\begin{enumerate}
	\item $\Inv(\Brm{\rm S}_n,{\rm M}) = {\rm M}^{\bullet}(k) \oplus \alpha_1\cdot {\rm M}^{\bullet}(k)_2 \oplus \ldots \oplus \alpha_{\left[ n/2 \right]} \cdot {\rm M}^{\bullet}(k)_2 $.
	\item $\Inv(\Brm {\rm PGL}_2,\M) = {\rm M}^{\bullet}(k) \oplus w_{2}\cdot {\rm M}^{\bullet}(k)_2 $.
\end{enumerate}
In particular, all non-trivial invariants are of $2$-torsion.
\end{prop}
\begin{proof}
	We start by proving (1).
	
	Let $H\simeq (\ZZ/2\ZZ) ^{\times m}$ be the subgroup of ${\rm S}_n$ generated by the transpositions $(1,2)$, $(3,4),\ldots, (n-1,n)$ for $n$ even, and by $(1,2),\ldots, (n-2,n-1)$ for $n$ odd.
	We know from Lemma \ref{lm:Inv mu2r} that:
	\[ \Inv(\Brm H,\M)\simeq \M^{\bullet}(k)\oplus \M^{\bullet}(k)_2\cdot \gamma_1 \oplus \cdots \M^{\bullet}(k)_2\cdot \gamma_m \]
	Observe that $S_m$ acts on $\Inv(\Brm H,\M)$ by permuting $\gamma_1,\dots,\gamma_m$.
	By Theorem \ref{thm:split}, the pullback of invariants along the morphism $\Brm H\to \Brm {\rm S}_n$ induces an injective morphism
	\begin{equation}\label{eq:pullback map}
	 \Inv(\Brm {\rm S}_n,\M)\longrightarrow \Inv(\Brm H,\M)^{S_m}
	\end{equation}
	The group on the right hand side is equal to
	\[ \M^{\bullet}(k)\oplus \M^{\bullet}(k)_2\cdot \alpha_1 \oplus \cdots \oplus \M^{\bullet}(k)_2 \cdot \alpha_m \]
	where $\alpha_i$ is the elementary symmetric polynomial of degree $i$ in the $\gamma_1,\dots,\gamma_m$.
	
	The surjectivity of (\ref{eq:pullback map}) is proved in \cite[Thm. 25.6.(1)]{GMS} for $\K_2$. We want to extend this result to the general case. As in the case of ${\rm O}_n$ we have a commutative square
	
\[ \xymatrix{
			\Inv({\rm BS}_n,\K_2) \otimes {\rm M}(k)_2, \ar[r]^{\cdot} \ar[d] & \Inv({\rm BS}_n, {\rm M})_2 \ar[d] \\
		 \Inv(\Brm\mu_2^m, \K_2)^{{\rm S}_m} \otimes {\rm M}(k)_2 \ar[r]^{\cdot} &  \Inv(\Brm\mu_2^m, {\rm M})^{{\rm S}_m}_2 }
		\]	
		
Again, the left vertical map and lower horizontal map are surjective, proving our claim.		
	
	
	
	
To prove the second statement, note that in characteristic different from two we have an isomorphism ${\rm PGL}_2 \simeq {\rm SO}_3$. For $n$ odd, we also have that ${\rm SO}_n = {\rm O}_n \times \mu_2$. Then we can conclude by combining Proposition \ref{prop:Amun} with Proposition \ref{prop:Inv On}.

\end{proof}

\begin{rmk}
One can take the reasoning in Proposition \ref{prop:Inv On} and Proposition \ref{prop:Inv Sn and PGL2} further to get the following criterion:

Let $G$ be a finite group acting on $I=\lbrace 1,\ldots, r\rbrace$, and assume for every ${\rm M}$ we have an injection 
\[\Inv(\Xcal, {\rm M}) \subseteq \Inv(\mu_\ell^r, {\rm M})^G.
\]
Then $\Inv(\Xcal, {\rm M})=\Inv(\mu_\ell^r, {\rm M})^G$ if and only if $\Inv(\Xcal, \K_{\ell})^G=\Inv(\mu_\ell^r, \K_{\ell})^G$.
\end{rmk}

As an application of Prop. \ref{prop:Inv Sn and PGL2}, we compute the Brauer groups of ${\rm B}{\rm S}_n$ and ${\rm BPGL}_2$.

\begin{lm}\label{BrFinite}
Let $\Mcal$ be a Deligne-Mumford stack, finite and smooth over $k$. Then ${\rm Br}(\Mcal)={\rm Br}'(\Mcal)$.
\end{lm}
\begin{proof}
By \cite{Gab81}*{Ch.II, Lemma 4}, given a surjective finite locally free map $Y \xrightarrow{f} X$, if $\alpha \in {\rm Br}'(X)$ and $f^*\alpha \in {\rm Br}(Y)$ then $\alpha \in {\rm Br}(X)$. The result is proven in the setting of strictly ringed topoi, so in particular it holds for Deligne--Mumford stacks. 

Then we can just apply it to $X \rightarrow \Mcal$, where $X$ is an affine cover of $\Mcal$, and use the fact that for zero dimensional schemes it is always true that ${\rm Br}={\rm Br}'$.
\end{proof}

\begin{cor}
We have
\[^c{\rm Br}({\rm B}{\rm S}_n) = \hspace{1pt} ^c{\rm Br}(k) \oplus {\rm H}^1(k, \ZZ/2\ZZ) \oplus \ZZ/2\ZZ,\quad ^c{\rm Br}({\rm B}{\rm PGL}_2)=\hspace{1pt}^c{\rm Br}(k)\oplus \ZZ/2\ZZ\]
\end{cor}
\begin{proof}
The formulas for the cohomological Brauer group are immediate from the description of the generalized cohomological invariants given in Proposition \ref{prop:Inv Sn and PGL2}. 

For ${\rm B}{\rm S}_n$, we know that the Brauer group is equal to the cohomological Brauer group by the Lemma above. 

For ${\rm B}{\rm PGL}_2$, note that the universal conic $\mathcal{C}=\left[P^1/{\rm PGL}_2\right]$ induces a nonzero element $\lbrace \mathcal{C} \rbrace$ in the Brauer group which is trivial when pulled back to the base field. This implies that $\lbrace \mathcal{C} \rbrace = w_2 + \alpha_0 $, where $\alpha_0$ belongs to ${\rm Br}(k)$, showing that ${\rm Br}({\rm B}{\rm PGL}_2)={\rm Br}'({\rm B}{\rm PGL}_2)$.
\end{proof}

\section{The moduli stacks of hyperelliptic curves}\label{sec: pres Hg}



We briefly review Arsie and Vistoli's presentation of the moduli stack $\Hcal_g$, and recall some results from \cite{PirCohHypEven} and \cite{DilPir} that will be needed later. Let $n$ be an even positive integer, and consider the affine space $\bA^{n+1}$, seen as the space of binary forms of degree $n$. There are two different natural actions on this space. 

\begin{itemize}
\item An action of ${\rm GL}_2$ given by \[A\cdot f(x_0,x_1) = {\rm det}(A)^{n/2-1}f(A^{-1}(x_0,x_1)).\]
\item An action of ${\rm PGL}_2 \times \Gm$ given by \[(\left[A\right],t)\cdot f(x_0,x_1)={\rm det}(A)^{n/2}t^{-2}(f(A^{-1}(x_0,x_1))).\]
\end{itemize}

The open subset of square-free forms inside $\bA^{n+1}$ is $G$-invariant. We will denote it by $\bA^{n+1}_{\rm sm}$.

\begin{thm}\label{thm:presentation}
When $g\geq 2$ is even, we have an isomorphism $\left[\bA^{2g+3}_{\rm sm}/{\rm GL}_2\right] \simeq \Hcal_g$.

When $g\geq 3$ is odd, we have an ismorphism $\left[\bA^{2g+3}_{\rm sm}/{\rm PGL}_2\times \Gm\right] \simeq \Hcal_g$.
\end{thm}
\begin{proof}
This is proven in \cite[4.7]{ArsVis}.
\end{proof}

When no confusion is possible, we will write $G$ for either ${\rm GL}_2$ or ${\rm PGL}_2\times \Gm$. Our computation will be for the most part done on the projectivizations \[P^n = \left( \bA^{n+1}\smallsetminus\{0\} \right)/\Gm\] with the induced action of $G$. 

The following $G$-invariant stratification will be crucial. Let $\Delta_{i}^n \subset P^n$ be the closed subscheme of $P^n$ whose points are forms of degree $n$ which are divisible by the square of a form of degree at least $i$, with the reduced subscheme structure. Then 
\[ P^n \supset \Delta^n_1 \supset \ldots \supset \Delta_{n/2}^n \]
is a $G$-invariant stratification of $P^n$.

 We define $\Delta_{\left[i\right]}^{n}$ as the subscheme of forms of degree $n$ divisible by the square of a form of degree $i$ but not higher, i.e. $\Delta_{\left[i\right]}^{n}=\Delta_{i}^n \setminus \Delta_{i+1}^n$, and similarly given $j > i$ we define $\Delta_{\left[ i,j\right]}^n = \Delta_{i}^n \setminus \Delta_{j+1}^n$. Finally, in keeping with this notation, we define the open subsets $P^{n}_{\rm sm}=P^{n}_{\left[0\right]}=P^{n}\setminus \Delta_{1}^n$ and $P^{n}_{\left[0,i\right]}=P^n \setminus \Delta_{i+1}^n$.
 
There is a natural map $P^{n-2r} \times P^r \rightarrow \Delta^{n}_{r}$ given by $(f,g) \rightarrow fg^2$. Checking that the map is equivariant with respect to the action of $G$ is easy. Note that if we restrict the map to $P^{n-2r}_{\rm sm} \times P^r$ the image is exactly $\Delta_{\left[r \right]}^n$.

\begin{prop}[\cite{PirCohHypEven}*{Prop. 3.3}]
The map $P^{n-2r}_{\rm sm} \times P^r \rightarrow \Delta^n_{\left[ r \right]}$ is an equivariant universal homeomorphism.
\end{prop}

Consider the map $(P^1)^n \rightarrow P^n$ given by $(l_1, \ldots, l_n) \mapsto l_1\ldots l_n$. If we restrict the map to $P^n_{\rm sm}$, it becomes a ${\rm S}_n$-torsor, thus inducing a map $P^n_{\rm sm} \rightarrow {\rm B}{\rm S}_n$. 

Another way of seeing this is the following: consider the stack ${\rm \'et}_n$ of \'etale algebras of degree $n$. There is a natural isomorphism ${\rm \'et}_n \simeq {\rm B}{\rm S}_n$. Given a point $f:S \rightarrow P^n_{\rm sm}$, we get an \'etale algebra by taking the zero locus of the form $f$ over $P^1_S$. One can easily check that these two maps coincide.

Now, the stack $\left[ P^n_{\rm sm}/G \right]$ parametrizes families of conics $C'/S$ (trivial families if $G={\rm GL}_2$) equipped with a line bundle $L$ of vertical degree $n/2$ and a subscheme $W_{C'}$ of codimension $1$, \'etale on the base, whose associated divisor is in the linear series of $L^{\otimes 2}$. 

Taking the subscheme $W_{C'}/S$ we get a map to ${\rm \'et}_n = {\rm B}{\rm S}_n$. It's easy to check that this map provides a factorization to the maps above, and that moreover it extends to $\bA^{n+1}_{\rm sm}$ and $\left[ \bA^{n+1}_{\rm sm}/G \right]$. When we take the stack $\Hcal_g$ this map is precisely the map that sends an hyperelliptic curve over $S$ to its Weierstrass divisor, seen as an $S$-scheme.

Write a form $f=a_0x_0^n + \ldots + a_nx_1^n$. The restriction of the map $P^n_{\rm sm} \rightarrow {\rm B}{\rm S}_n$ to the complement of the hyperplane $a_0=0$ comes from the generically free ${\rm S}_n$--representation on $\bA^n$ given by permuting the coordinates. This implies that the map is smooth-Nisnevich, so it induces an injective map on cohomological invariants.

\begin{prop}[\cite{DilPir}*{Prop. 1.1- Cor. 1.2}]
The pullback map $\Inv({\rm B}{\rm S}_n,\M)\rightarrow \Inv(\Mcal,\M)$, where $\Mcal$ is any of the above stacks, is injective.
\end{prop}

\section{The generalized cohomological invariants of $\Hcal_g$}\label{sec: Coh Inv Hg}

We are ready to compute the cohomological invariants of $\Hcal_g$ with coefficients in a general $\ell$-torsion cycle module $\M$. The specialization of this computation to degree two invariants with coefficients in ${\rm H}_{\mu_{\ell}^{\vee}}$ will give us the result on the Brauer groups.

From now on, the cycle module $\M$ will be assumed to be always of $\ell$-torsion.

It is worth noting that computing only the degree two invariants does not seem any easier than computing (almost) the full group in this case, but in more general situations it might be possible to obtain the result in low degrees even if the full ring seems too hard to approach.

First we need a few more lemmas from \cites{PirCohHypEven, PirCohHypThree}. Recall that 
\[{\rm CH}^{*}_{{\rm GL}_2}(P^n)=\ZZ\left[t,\lambda_1, \lambda_2\right]/(R_n(t,\lambda_1,\lambda_2)) \]
where $\lambda_1, \lambda_2$ are the Chern classes of the standard representation of ${\rm GL}_2$, $t$ is the first Chern class of $\OO_{P^n}(-1)$ and $R_n(t,\lambda_1,\lambda_2)$ is a polynomial of degree $n$. We also have:
\[{\rm CH}^{*}_{{\rm PGL}_2\times \Gm}(P^n)=\ZZ\left[t,s,c_2, c_3\right]/(T_n(t,c_2,c_3),2c_3)\]
where $t$ is the first Chern class of $\OO_{P^n}(-1)$, $s$ is the first Chern class of the standard representation of $\Gm$, $c_2,c_3$ are the second and third Chern classes of the three dimensional representation of ${\rm PGL}_2$ coming from the isomorphism ${\rm PGL}_2 \simeq {\rm SO}_3$.

We should also note that if $\Gm$ acts trivially on $X$ then $A^{*}_{\Gm}(X,\M)=A^{*}(X,\M)\otimes \ZZ\left[s\right]$. In particular, for our computations we can often consider $G={\rm PGL}_2$ rather than ${\rm PGL}_2\times \Gm$.

\begin{lm}
The class of $\Delta_{1}^n$ in ${\rm CH}^{1}_{{\rm GL}_2}(P^n)$ is $2(n-1)(t + n\lambda_1)$. The class of $\Delta_{1}^n$ in ${\rm CH}^{1}_{{\rm PGL}_2}(P^n)$ is $2(n-1)t$.
\end{lm}
\begin{proof}
For $G={\rm GL}_2$, the statement is proved in \cite{PirCohHypEven}*{Prop. 4.3}. For $G={\rm PGL}_2$, see \cite{DilChowHyp}*{Prop. 5.2}, albeit there is a mistake in the statement of the result, as $4(n-2)h_n$ should be replaced by $(4n-2)h_n$ (note that the $n$ in the statement of \cite{DilChowHyp}*{Prop. 5.2} would be $n/2$ in this Lemma's notation).
\end{proof}

We begin by dealing with the case of odd torsion, which is easy due to the following Lemmas. 

\begin{lm}
Consider the natural action of ${\rm PGL}_2$ on $P^1$. We have
\[A^{*}_{{\rm PGL}_2}(P^1,\M)\simeq A^{*}_{\Gm}({\rm Spec}(k),\M) = \ZZ\left[t \right]\otimes \M^{\bullet}(k).\]
\end{lm}
\begin{proof}
We have \[A^*_{G}(P^1,\M) \simeq A^{*}_{H}(\Spec(k),\M)\]
 where $H$ is the stabilizer of a point in $P^1$, which is isomorphic to a semidirect product $\Gm \ltimes \mathbb{G}_a$. Then by \cite[8]{PirCohHypThree} we have \[A^{*}_{H}(\Spec(k),\M)\simeq A^{*}_{\Gm}(\Spec(k),\M)\]
 which allows us to conclude immediately.
\end{proof}

\begin{lm}\label{P1P1}
Let $\ell$ be odd. Then the pullback map $A^0_G(P^n_{\rm sm},\M) \rightarrow A^{0}_G(P^n_{\rm sm} \times P^1 \times P^1,\M)$ is surjective.
\end{lm}

We postpone the proof this Lemma, as it involves an argument similar to the one used in the next Lemma and in the following Proposition.

\begin{lm}\label{lm:Delta1 trivial}
Suppose $\ell=p^n$, with $p$ a prime number different from $2$. Then $A^0_{G}(\Delta_{1}^n,{\M})=\M^{\bullet}(k)$.
\end{lm}
\begin{proof}
  As $A^{0}_{G}(\Delta^n_{1},\M)$ is isomorphic to $A^{0}_{G}(\Delta^n_{\left[ 1,2 \right]},\M)$ (because $\Delta^n_{3}$ has codimension two in $\Delta^n_{1}$) we can compute it using the following exact sequence:

 \begin{center}
 $ 0 \rightarrow A^{0}_{G}(\Delta^n_{\left[ 1,2\right]},\M) \rightarrow A^{0}_{G}(\Delta^n_{\left[1\right]},\M)  \xrightarrow{\partial} A^{0}_{G}(\Delta^n_{\left[2\right]},\M)$
 \end{center}
 
We want to prove that the kernel of $\partial$ is equal to $\M^{\bullet}(k)$. This will then imply that the image of $A^{0}_{G}(\Delta^n_{\left[1,2\right]},\M)$ must be equal to $\M(k)$, and thus it must be trivial.
 
The map $P^{n-2}_{\left[0,1\right]} \times P^{1} \xrightarrow{\pi} \Delta^n_{\left[1,2\right]} $ yields the following commutative diagram with exact columns:

\begin{center}
$\xymatrixcolsep{3pc}
\xymatrix{ A^{0}_{G}(P^{n-2}_{\left[0,1\right]} \times P^1 ,\M) \ar@{->}[r]^{\pi_{*}} \ar@{->}[d] & A^{0}_{G}( \Delta^n_{\left[1,2\right]},\M)   \ar@{->>}[d] \\
A^{0}_{G}(P^{n-1}_{\left[0\right]} \times P^1,\M) \ar@{^{(}->>}[r]^{\pi_{*}} \ar@{->}[d]^{\partial_1} &  A^{0}_{G}(\Delta^n_{\left[1\right]},\M) \ar@{->}[d]^{\partial}\\
 A^{0}_{G}(\Delta^{n-2}_{\left[1\right]}\times P^1,\M) \ar@{->}[r]^{\pi_{*}} & A^{0}_{G}(\Delta^n_{\left[2\right]},\M)  } $
\end{center}

 The second horizontal map is an isomorphism because $\pi_*$ is a universal homeomorphism when restricted to $\Delta^n_{\left[ 1\right]}$.
 
  The kernel of $\partial_{1}$ is trivial because $A^{0}_{G}(P^{n-2}_{\left[0,1\right]} \times P^1,\M)$ is trivial, as $A^{0}_{G}(P^{n-2} \times P^1,\M)$ is trivial by the projective bundle formula and $\Delta^{n-2}_{2} \times P^r$ has codimension $2$. 
 
  We claim that the third horizontal map is injective, implying that the kernel of $\partial$ must be trivial too.
 
 Let $\psi$ be the map from $P^{n-4}_{\left[0\right]} \times P^{1} \times P^1$ to $P^{n-4}_{\left[0\right]} \times P^{2}$ sending $(f,g,h)$ to $(f,gh)$. We have a commutative diagram:
 
\begin{center}
$\xymatrixcolsep{5pc}
\xymatrix{ P^{n-4}_{\left[0\right]} \times P^1 \times P^{1}  \ar@{->}[r]^{\pi_1} \ar@{->}[d]^{\psi} & \Delta^{n -2}_{\left[1\right]}  \times P^1 \ar@{->}[d]^{\pi}\\ 
P^{n-4}_{\left[0\right]}\times P^{2}  \ar@{->}[r]^{\pi_2} & \Delta^n_{\left[2\right]}} $
\end{center}

Where $\pi_1$ and $\pi_2$ are defined respectively by $(f,g,h) \rightarrow (fg^2,h)$ and $(f,g) \rightarrow (fg^2)$. The maps $\pi_1$ and $\pi_2$ are universal homeomorphisms, so the pushforward maps $(\pi_1)_*,(\pi_2)_*$ are isomorphisms. Then if we prove that $\psi_*$ is injective $\pi_*$ will be injective too. Consider this last diagram:
 
\begin{center}
$\xymatrixcolsep{5pc}
\xymatrix{ P^{n-4}_{\left[0\right]}\times P^{1} \times P^1 \ar@{->}[dr]^{p_1} \ar@{->}[d]^{\psi}\\ 
P^{n-4}_{\left[0\right]}\times P^{2}  \ar@{->}[r]^{p_2} & P^{n-4}_{\left[0\right]}}$
\end{center}
The pullbacks along $p_1$ and $p_2$ are both surjective, implying that the pullback along $\psi$ is surjective. We have $\psi_* ( \psi^* \alpha) = \mathrm{deg}(\psi) \alpha$ by the projection formula. Then as the degree of $\psi$ is $2$, $\psi_*$ is injective.
 
\end{proof}

The Lemma tells us that when $\M$ is of odd torsion, the group $A^0_G(P^n_{\rm sm},\M)$ must fit in the following exact sequence
\[0 \rightarrow \M^{\bullet}(k)=A^0_G(P^n,\M) \rightarrow  A^0_G(P^n_{\rm sm},\M) \rightarrow \M^{\bullet}(k)=A^0_G(\Delta^n_{1},\M) \rightarrow A^1_G(P^n,\M).\]

So, roughly speaking, it can contain at most an additional copy of the cohomology of the base field, shifted in degree by one. This is of course imprecise at this point, as the sequence may not split. 

In the next Proposition we compute the kernel of the last map, and moreover we will show that the group $ A^0_G(P^n_{\rm sm},\M)$ splits as a direct sum of a copy of the trivial elements coming from $A^0_G(P^n,\M)$ and the kernel. This will allow us to easily compute the cohomological invariants of $\Hcal_g$ with coefficients in $\M$.

\begin{prop}\label{prop:CohInvHgodd}
Assume $\ell$ is odd. Let $\ell'$ be the greatest common divisor of $\ell$ and $2g+1$. Then \[\Inv(\Hcal_g, \M) \simeq \M^{\bullet}(k) \oplus \M^{\bullet}(k)_{\ell'}[1].\]
In particular, if $\ell$ is coprime with $2g +1$, the cohomological invariants of $\Hcal_g$ with coefficients in $\M$ are trivial.
\end{prop}
\begin{proof}
We begin by computing the invariants of $\left[ P^{2g+2}_{\rm sm}/ G \right]$.

First consider the case of $\M=\K_{\ell}$. Consider the exact sequence
\[
\begin{tikzcd}
  0 \rar &  A^0_{G}(P^{2g+2},\ZZ/p^n\ZZ) \rar
             \ar[draw=none]{d}[name=X, anchor=center]{}
    & A^{0}_G(P^{2g+2}_{\rm sm},\ZZ/p^n\ZZ) \ar[rounded corners,
            to path={ -- ([xshift=2ex]\tikztostart.east)
                      |- (X.center) \tikztonodes
                      -| ([xshift=-2ex]\tikztotarget.west)
                      -- (\tikztotarget)}]{dll}[at end]{} \\      
   A^{0}_G(\Delta^{2g+2}_{1},\ZZ/p^n\ZZ) \rar & A^{1}_G(P^{2g+2},\ZZ/p^n\ZZ) \rar & A^{1}_G(P^{2g+2}_{\rm sm},\ZZ/p^n\ZZ) 
\end{tikzcd}
\]
We have $A^{0}_G(\Delta^{2g+2}_{1},\K_{\ell})=\K_{\ell}^{\bullet}(k)$, and the pushforward map is just multiplication by $\lbrace \Delta^{2g+2}_{1} \rbrace$. Assume first that $\ell$ divides $2g+1$, so that the image of $\lbrace \Delta^{2g+2}_{1} \rbrace$ is zero. Then if we pick an inverse image $\gamma$ of $1$ the submodule $\gamma \cdot \K_{\ell}^{\bullet}(k)$ maps isomorphically to $A^{0}_G(\Delta_{1,n},\K_{\ell})$, splitting the exact sequence.

Now let $\ell$ be general and let $\ell'$ be the greatest common divisor of $\ell$ and $2g+1$. Let $\gamma \in A^0_G(P^{2g+2}_{\rm sm},\K_{\ell'})$ be as above.

Then for any $x \in \K_{\ell}^{\bullet}(k)_{\ell'}$ the element $\gamma \cdot x$ belongs to $A^0(P^{2g+2}_{\rm sm},\K_{\ell})$, and moreover the boundary map sends $\gamma \cdot\K_{\ell}^{\bullet}(k)_{\ell'}$ to $\K_{\ell}^{\bullet}(k)_{\ell'} \subset A^0(\Delta^{2g+2}_{1},\K_\ell)$, which is exactly the kernel, splitting the exact sequence.

For a general $\ell$-torsion cycle module $\M$, take $\gamma$ as above. It's easy to see that we have $A^{0}_G(P^{2g+2}_{\rm sm},\M) = \gamma \cdot \M^{\bullet}(k)_{\ell'}$.

The next step consists of verifying that the $\Gm$-torsor $\Hcal_g \rightarrow \left[ P^{2g+2}_{\rm sm}/ G \right]$ does not generate any new invariant. This can be done as in the proof of Theorem \ref{thm:CohInvHg2}, and it is an easy consequence of the fact that the first Chern class of the line bundle coming from the torsor does not annihilate any element in $\M^{\bullet}(k)$.
\end{proof}

\begin{proof}[Proof of Lemma \ref{P1P1}]
First note that if $G$ is ${\rm GL}_2$ or in the non-equivariant case the statement is obvious by the projective bundle formula. 

Now consider the case where $G={\rm PGL}_2$. We have \[A^0_{G}(P^n_{\rm sm}\times P^{1} \times P^1,\M) \simeq A^{0}_{H}(P^{n}_{\rm sm}\times P^{1},\M)\]
where $H$ is the stabilizer of a point in $P^1$, which is isomorphic to a semidirect product $\Gm \ltimes \mathbb{G}_a$. As $\mathrm{H}$ is a special group the pullback \[A^{0}_{H}(P^{n}_{\rm sm}\times P^{1},\M) \rightarrow  A^{0}(P^{n}_{\rm sm}\times P^{1}, \M)=A^{0}(P^{n}_{\rm sm}, \M)\] has to be injective. We claim that the pullback $A^0_G(P^n_{\rm sm},\M) \rightarrow A^0(P^n_{\rm sm},\M)$ is surjective, which would allow us to conclude immediately.

One can use the same techniques used to prove Lemma \ref{lm:Delta1 trivial} to easily show that when $p \neq 2$ the non-equivariant group $A^{0}( \Delta^n_{1},\M)$ is trivial, and thus $A^{0}(P^n_{\rm sm},\M)$ is either trivial or generated by $1$ and an element in degree one corresponding to the equation for $\Delta_{1}^{n}$, multiplied by the submodule of $\M^{\bullet}(k)$ which annihilates the class of $\Delta_{1}^{n}$ in $A^1(P^{n},\M)$.

In the latter case, consider the following commutative diagram induced by the pullback from equivariant to non-equivariant Chow groups with coefficients
 
 \begin{center}
$\xymatrixcolsep{3pc}
\xymatrix{  A^{0}_{G}(P^{n},\M) \ar@{->}[r] \ar@{->}[d] & A^{0}(P^{n},\M)   \ar@{->}[d] \\
A^{0}_{G}(P^{n}_{\rm sm},\M) \ar@{->}[r] \ar@{->}[d]^{\partial} &  A^{0}(P^{n}_{\rm sm},\M) \ar@{->}[d]^{\partial}\\
 A^{0}_{G}(\Delta_{1}^n,\M) \ar@{->}[r] \ar@{->}[d] & A^{0}(\Delta_{1}^n,\M)  \ar@{->}[d]\\
  A^{1}_{G}(P^{n},\M) \ar@{->}[r]  & A^{1}(P^{n},\M)} $
\end{center}

Both the top and the bottom horizontal map are isomorphisms, as one can see using the fact that the groups on the top row are trivial and the groups on the bottom row are given by $c_1^G(\mathcal{O}_{P^{n}}(-1))\cdot \M^{\bullet}(k)$.

Moreover $A^{0}(P^{n}_{\rm sm},\M)$ is generated as by $1 \cdot \M^{\bullet}(k)$ and an element $\alpha \cdot \M^{\bullet}(k)$, where $\alpha$ is an element such that $\partial(\alpha)=1$ in $A^{0}_{G}(\Delta^n_{1},\M)$, with no additional relations.

The third horizontal map maps $1 \in A^{0}_{G}(\Delta_{1}^n,\M )$ to $1 \in A^{0}(\Delta_{1}^n,\M)$, which shows that an element $\tau$ of degree zero maps to zero in the equivariant group $ A^{1}_{G}(P^{n},\M)$ if and only if it maps to zero in $ A^{1}(P^n,\M)$. Then there must be an element \[\alpha' \in  A^{0}_{G}(P^{n}_{\rm sm},\M)\] which maps to $\alpha \in  A^{0}(P^{n}_{\rm sm},\M)$, thus the pullback $A^{0}_{G}(P^{n}_{\rm sm},\M) \rightarrow A^0(P^{n}_{\rm sm},\M)$ is surjective. This concludes the proof.
\end{proof}

The even case is much more complicated. We first compute the invariants with coefficients in $\K_{2^r}$ of $\left[ P^{2g+2}_{\rm sm} / G \right]$. Our computation is based on the fact that we already know that the cohomological invariants of ${\rm S}_n$ inject into those of $\left[ P^{2g+2}_{\rm sm} / G \right]$. Using this, we inductively show that ``there is no more room'' and we have found all the invariants. This is done for $r=1$ in \cite{DilPir}.

\begin{lm}
We have:
\begin{itemize}
\item if $g$ is even, then
\[\Inv(\left[ P^{2g+2}_{\rm sm} / G \right], \K_{2^r}) \simeq \Inv(\Brm {\rm S}_{2g+2},\K_{2^r}).\]

\item If $g$ is odd, then
\[\Inv(\left[ P^{2g+2}_{\rm sm} / G \right], \\K_{2^r}) \simeq \Inv(\Brm {\rm S}_{2g+2},\K_{2^r}) \oplus \K_{2^r}^{\bullet}(k)_2\left[ 2 \right]\]
where the copy of $\K_{2^r}^{\bullet}(k)_2\left[ 2 \right]$ comes from the cohomological invariants of ${\rm PGL}_2$.
\end{itemize}
\end{lm}
\begin{proof}
Consider the exact sequence
\[A^{0}_G(P^n,\K_{2^r}) \hookrightarrow A^{0}_G(P^n_{\rm sm},\K_{2^r}) \xrightarrow{\partial} A^{0}_G(\Delta^n_{1},\K_{2^r}) \subset A^{0} _G(P^{n-2}_{\rm sm}\times P^1, \K_{2^r})\]
First consider the case $r=1$.  We will proceed by induction on $n$. If we assume the result for $n-2$, we have that $ A^{0} _G(P^{n-2}_{\rm sm}\times P^1, \K_{2^r})=\Inv(\Brm {\rm S}_{n-2},\K_{2})$. Note that if $G={\rm PGL}_2$ the invariant $w_2$ in $A^{0} _G(P^{n-2}_{\rm sm}\times P^1, \K_{2^r})$ is killed by the $P^1$. The cokernel of the first map contains the submodule of non-trivial invariants of ${\rm S}_n$. Thus we have a map
\[\Inv(\Brm {\rm S}_n,\K_{2}) / \K_{2}^{\bullet}(k) \rightarrow \Inv(\Brm {\rm S}_{n-2},\K_{2})_2\]
which lowers degree by one. Comparing the generators of the image with the generators of $A^{0} _G(P^{n-2}_{\rm sm}\times P^1, \K_{2^r})$ degree by degree shows that the map is surjective, and in particular there cannot be any additional element in $A^{0}_G(P^n_{\rm sm},\K_{2})$, yielding the result.

Now pick $r > 1$. We will proceed by induction on $n$ even. If we assume the result for $n-2$, we have that $ A^{0} _G(P^{n-2}_{\rm sm}\times P^1, \K_{2^r})=\Inv(\Brm {\rm S}_{n-2})$. Note that if $G={\rm PGL}_2$ the invariant $w_2$ is killed by the $P^1$. Moreover, note that as the image of $\left[ \Delta_{1,n} \right]$ is divisible by $2$ but not by $4$, the intersection of the kernel of $A^{0}_{G}(\Delta_{1,n},\K_{2^r}Z) \rightarrow A^{1}_G(P^n,\K_{2^r})$ and the subring of trivial invariants of ${\rm S}_{n-2}$ is exactly the two-torsion.

On the other hand, the cokernel of the inclusion of $A_G^0(P^n,\K_{2^r})$ into the group $A_G^0(P^n_{\rm sm},\K_{2^r})$ is isomorphic to the subgroup of non-trivial invariants of ${\rm S}_n$. Thus we have a map
\[\Inv(\Brm {\rm S}_n,\K_{2^r}) / \K_{2^r}^{\bullet}(k) \rightarrow \Inv(\Brm {\rm S}_{n-2},\K_{2^r})_2\]
which lowers the degrees by one. We claim this map is surjective. Let 
\[\alpha = \tau_0 + \alpha_{1}\tau_1 + \ldots +\alpha_{n/2-1}\tau_{n/2-1} \]
with $\tau_0, \ldots, \tau_{n/2-1} \in \K_{2^r}^{\bullet}(k)_2$ elements of $\Inv(\Brm {\rm S}_{n-2},\K_{2^r})$. Let $\zeta_{1}, \ldots, \zeta_{n}$ be elements of $A^0_{G}(P^{n}_{\rm sm}, \K_{2})$ such that $\partial(\zeta_1)=1, \ldots, \partial(\zeta_n)=\alpha_{n/2-1}$.
 
Then by Lemma \ref{InvImage} the element
\[\zeta = \zeta_{1} \tau_0 + \ldots + \zeta_{n/2} \tau_{n/2-1}\]
belongs to $A^0_{G}(P^{n}_{\rm sm}, \K_{2^r})$. Now note that by the compatibility of the boundary map with the morphism of cycle modules $\K_{2^r} \rightarrow \K_{2}$ the restriction modulo two of $\partial(\zeta_i)$ is equal to $\alpha_{i-1}$. But then $\alpha_{i-1}-\partial(\zeta_{i})$ is a multiple of $2$, proving that $\partial(\zeta)=\alpha$. This concludes the proof. 
\end{proof}

\begin{cor}
Assume that $\ell$ is a power of $2$. If $g$ is even, the cohomological invariants with coefficients in $\M$ of $\left[ P^{2g+2}_{\rm sm} / G \right]$ are isomorphic to the cohomological invariants of ${\rm S}_{2g+2}$.

If $g$ is odd, the cohomological invariants with coefficients in $D$ of $\left[ P^{2g+2}_{\rm sm} / G \right]$ are a direct sum of the cohomological invariants of ${\rm S}_{2g+2}$ and a copy of $\M^{\bullet}(k)_2\left[ 2 \right]$ coming from ${\rm PGL}_2$.
\end{cor}
\begin{proof}
Consider the exact sequence
\[
A^{0}_G(P^n,\M) \hookrightarrow A^{0}_G(P^n_{\rm sm},\M) \xrightarrow{\partial} A^{0}_G(\Delta^n_{1},\M) \subset A^{0} _G(P^{n-2}_{\rm sm}\times P^1, \M)
\]
As in the lemma above, we proceed by induction on even $n$. If the result is true for $n-2$, we reduce to show that the map
\[\partial:\alpha_1\cdot \M^{\bullet}(k)_2 \oplus \ldots \oplus \alpha_{n/2}\cdot \M^{\bullet}(k)_2 \rightarrow \M^{\bullet}(k)_2 \oplus \ldots \oplus \alpha_{n/2-1}\cdot\M^{\bullet}(k)_2 \]
is surjective. This is done exactly as in the previous lemma, using the structure of $A^{0}_G(P^n_{\rm sm},\K_{\ell})$-module of $A^{0}_G(P^n_{\rm sm},\M)$.
\end{proof}

\begin{lm}
Consider the line bundle $\mathcal{L}$ associated to the $\Gm$-torsor \[\Hcal_g \rightarrow \left[ P^{2g+2}_{\rm sm}/G\right].\] The first Chern class $c_1(\mathcal{L})$ is equal to:
\begin{itemize}
\item $g\lambda_1-t$ if $g$ is even.
\item $t-2s$ if $g$ is odd.
\end{itemize}
\end{lm}
\begin{proof}
The first formula is proven in \cite{EF}*{Lemma 3.2}, the second formula in \cite{FV}*{Eq. 3.2}.
\end{proof}
\begin{thm}\label{thm:CohInvHg2}
Let $\ell$ be a power of $2$. For all $g$, there is a submodule ${\rm N}^{\bullet}_g(k)$ of $\M^{\bullet}(k)_2$ such that
\begin{enumerate}
\item If $g$ is even, there is an exact sequence \[0 \rightarrow \Inv(\Brm {\rm S}_n,\M) \rightarrow \Inv(\Hcal_g,\M) \rightarrow  {\rm N}^{\bullet}_g(k) \rightarrow 0\]
such that the inverse image of a non-zero element in ${\rm N}^{\bullet}_g(k)$ has degree at least $g+2$.
\item If $g$ is odd, let \[I_{g} = \alpha_2 \cdot \M^{\bullet}(k)_2 \oplus \ldots \oplus \alpha_n \cdot \M^{\bullet}(k)_2 \subset \Inv(\Brm {\rm S}_n,\M).\] 
There is an exact sequence
\[0 \rightarrow  \M^{\bullet}(k)_2 \oplus \M^{\bullet}(k)_4\!\left[1\right] \oplus I_g \oplus  \M^{\bullet}(k)_2\!\left[2\right] \rightarrow \Inv(\Hcal_g,\M) \rightarrow {\rm N}^{\bullet}_g(k) \rightarrow 0\]
such that the inverse image of a non-zero element in ${\rm N}^{\bullet}_g(k)$ has degree at least $g+2$. The $\M^{\bullet}(k)_4\!\left[1\right]$ is equal to $\alpha'_1 \cdot \M^{\bullet}(k)_4$, where $\alpha'_1$ is a square root of $\alpha_1$.
\end{enumerate}

\end{thm}
\begin{proof}
The map $\Hcal_g \rightarrow \left[ P^{2g+2}_{\rm sm}/G \right]$ is a $\Gm$-torsor. In particular, it is smooth Nisnevich, so the pullback on cohomological invariants is injective. We have to check whether there are invariants defined on $\Hcal_g$ that do not come from the base.

Let $\mathcal{L} \rightarrow  \left[ P^{2g+2}_{\rm sm}/G \right]$ be line bundle determined to the $\Gm$-torsor: its equivariant Chow groups with coefficients are isomorphic to those of $\left[ P^{2g+2}_{\rm sm}/G \right]$. The inclusion of the zero section of $\mathcal{L}$ gives us the following long exact sequence
\[0 \rightarrow A^0_G(P^{2g+2}_{\rm sm},\M) \rightarrow \Inv(\Hcal_g,\M) \xrightarrow{\partial} A^0_G(P^{2g+2}_{\rm sm},\M) \xrightarrow{c_1(\mathcal{L})}  A^1_G(P^{2g+2}_{\rm sm},\M) .  \]
Our goal is to understand the kernel of $c_1(\mathcal{L})$. We will do the computation in the odd genus case. The even case is much easier and follows from the same reasoning. 

In this case, we have $c_1(\mathcal{L})=t-2s$. The submodule of $A^1_G(P^{2g+2}_{\rm sm},\M)$ generated by $t$ and $s$ multiplied by $\M^{\bullet}(k)$ has the single relation $(4g+2)t=0$. Note that the relation implies that the submodule $t \cdot \M^{\bullet}(k)$ is of $2^r$ torsion and of $4g+2$ torsion. As $2g+1$ is odd, we get that the submodule $t \cdot \M^{\bullet}(k)$ is isomorphic to $\M^{\bullet}(k)/2\M^{\bullet}(k)$. 

On the other hand, the submodule $s \cdot \M^{\bullet}(k)$ has no additional relation, so the annihilator of $2s$ is exactly $\M^{\bullet}(k)_2$. Finally, it's easy to see that if $(t-2s)\cdot \tau =0$ for some $\tau$ in $\M^{\bullet}(k)$ then both $t\cdot \tau=0$ and $2s \cdot \tau=0$. 
This implies that any such $\tau$ must belong to both the $2$-torsion of $\M^{\bullet}(k)$ and $2\M^{\bullet}(k)$, i.e. $\tau$ is in $2\cdot \M^{\bullet}(k)_4$.

Now let us pick $\M=\K_4$. The kernel of $c_1(\mathcal{L})$, restricted to the elements of cohomological degree $0$, is generated by $2$. Given an inverse image of $2$ through $\partial$, we want to understand its relationship with the elements coming from $\left[P^{2g+2}_{\rm sm}/G \right]$. We  have ${\rm Inv}^1(\Hcal_g,K_4)={\rm H}^1(\Hcal_g,\mu_4)$. The latter has to surject onto ${\rm Pic}(\Hcal_g)_4=\ZZ/4\ZZ$. Comparing the two formulas we conclude that up to elements coming from the base field there must be an element $\alpha'_1 \in {\rm Inv}^1(\Hcal_g,\K_4)$ such that $2 \alpha'_1 = \alpha_1$ and $\partial \alpha'_1=2$.

For $\M$ a general $2^r$-torsion cycle module we can consider the submodule $\alpha'_1 \cdot \M^{\bullet}(k)_4$. It is immediate that the map $\partial$ sends it surjectively to $2 \cdot \M^{\bullet}(k)_4$ with kernel given exactly by $\alpha_1 \cdot \M^{\bullet}(k)_2$.  

Now consider an element \[\alpha=\tau_0 + \alpha_1 \tau_1 + \ldots + \alpha_{g} \tau_{g} + w_2 \sigma \in A^0_G(P^{2g+2}_{\rm sm},\M).\] We want to show that if  $c_1(\mathcal{L})\alpha = 0$ then we must have $\tau_1 = \ldots = \tau_g = \sigma = 0$. Note that the highest degree element $\alpha_{g+1}$ does not appear in the formula. Recall also that every $\tau_i$ is of $2$-torsion for $i>0$.

If $c_1(\mathcal{L})\alpha = 0 \in A^1_G(P^{n}_{\rm sm},\M)$, then the pullback of this element to the non-equivariant group $A^0(P^{n}_{\rm sm},\M)$ must be trivial as well. The pullback of $c_1(\Lcal)$ is equal to $t$, and $w_2$ goes to zero, so we get $t(\tau_0 +  \alpha_1 \tau_1 + \ldots + \alpha_{g} \tau_{g})=0$. 

For every $n\geq 4$ even, consider the morphism
\[ \Phi_n:A^0(P^n_{\rm sm},\M)\xrightarrow{\partial} A^0(\Delta_{\left[1\right],n},\M) \xrightarrow{\pi_*^{-1}} A^0(P^{n-2}_{\rm sm}\times P^1,\M)\simeq A^0(P^{n-2}_{\rm sm},\M)  \]
where $\pi:P^{n-2}_{\rm sm}\times P^1\to \Delta_{\left[1\right],n}$ maps $(f,g)$ to $fg^2$ and the last isomorphism is due to the projective bundle formula.
By construction we know that $\partial(\alpha_{i})=\alpha_{i-1}+\beta$, where $\beta$ belongs to the submodule generated by $\alpha_{i-2}, \ldots, \alpha_{1}, 1$. In particular,
\[ \Phi_{2g+2}(t\alpha)=t(\alpha' + \alpha_{g-1} \tau_g)\]
where $\alpha'$ is a combination of multiples of $1, \ldots, \alpha_{g-2}$.
After repeating this process $g-1$ times, we eventually end up with
$ t\tau_g = 0 $. As the image of $A^0(\Delta^1_{n},\M)$ in $A^1(P^n,\M)$ is generated by the pushforward of $1$, which is an even multiple of $t$, there are no additional relations in the submodule $t \M^{\bullet}(k)_2$. This implies that $\tau_g = 0$, thus $t(\tau_0 + \alpha_1 \tau_1 + \ldots + \alpha_{g-1} \tau_{g-1})=0$ in $A^1_G(P^{n}_{\rm sm},\M)$.

Applying the same argument several times, we deduce $\tau_i=0$ for $i>0$.  and $t\tau_0=0$, from which we also deduce that $\tau_0=2\tau_0'$.
In other terms, we have proved that $\alpha=\tau_0+\sigma w_2$, hence:
\[ 0=(t-2s)(\tau_0+w_2 \sigma )=(-2s+t)\tau_0+t w_2 \sigma  \]

 Note that the submodule $t (w_2 \cdot \M^{\bullet}(k)_2)$ has no additional relations and it is independent from $t \M^{\bullet}(k)_2$ and $s\M^{\bullet}(k)$ due to the projection formula and the description of the cohomological invariants of ${\rm PGL}_2$. Consequently, we need for $(t-2s)\tau_0$ and $t w_2 \sigma$ to be separately zero. We already know that the first requirement is equivalent to $\tau_0 \in 2\M^{\bullet}(k)_4$, and the second implies $\sigma=0$. 

Finally, consider a general element \[\alpha = \tau_0 + \alpha_1 \tau_1 + \ldots + \alpha_{g} \tau_{g} + w_2 \sigma + \alpha_{g+1} \tau_{g+1}.\] We may assume it is of homogeneous cohomological degree. If $c_1(\mathcal{L})\alpha=0$, either $\tau_{g+1}=0$, in which case $\alpha= \tau_0 \in 2\M^{\bullet}(k)_4$, or $\tau_{g+1} \neq 0$. Then we can consider the map $\Inv(\Hcal_g,\M) \rightarrow \M^{\bullet}(k)_2$ given by composing $\partial$ with the map sending $\alpha \in A^0_G(P^n_{\rm sm},\M)$ to $\tau_{g+1}$. The kernel of this map is exactly $\M^{\bullet}(k) \oplus \M^{\bullet}(k)_4\!\left[1\right] \oplus I_g \oplus  \M^{\bullet}(k)_2\!\left[2\right]$, and the inverse image of a nonzero element in $\M^{\bullet}(k)_2$ has degree at least $g+2$. This concludes the proof.
\end{proof}

\begin{rmk}\label{rmk:last inv}
Using the same techniques as in \cite{DilPir}*{Section 2.3}, we can conclude that when $g$ is even the module ${\rm N}^{\bullet}_g(k)$ is equal to $\M^{\bullet}(k)_2$ and explicitly describe the module structure:
\[\Inv(\Hcal_g,\M)= \alpha_1 \cdot \M^{\bullet}(k)_{4g+2} \oplus I_g \oplus \M^{\bullet}(k)_2\!\left[g+2\right] \]
where the last component is given by the element $\beta_{g+2}$ defined in \cite{DilPir}*{Section 2} and the $\Inv(\Hcal_g,\K_{\ell})$-module structure can be easily deduced from the multiplicative structure described in \cite{DilPir}*{Thm. 3.1}.

When $g$ is odd, as explained in \cite{DilPir}*{Section 4}, we can reach similar conclusions but only when $k$ is algebraically closed.
\end{rmk}

\begin{rmk}
Let $\overline{\Hcal}_g$ be the compactification of $\Hcal_g$ by means of stable hyperelliptic curves. Then, following \cite{DilPir}*{Appendix A}, we can easily conclude that the cohomological invariants $\Inv(\overline{\Hcal}_g,\M)$ are trivial whenever $\M$ is torsion.
\end{rmk}

\section{The Brauer group of $\Hcal_g$}\label{sec: Br Hg}

We are finally ready to describe the Brauer group of $\Hcal_g$. The following descriptions are immediate consequences of the description of $\Inv(\Hcal_g,\M)$ when $\M=\H_{\mu_\ell^{\vee}}$. Let $c={\rm char}(k)$ be the characteristic of the base field, and let $r_g$ be the remainder of $g$ mod $2$. Define $\ell_g=\ell_g(c)$ as the largest divisor of $2^{r_g}(4g+2)$ which is not divisible by $c$.

\begin{thm}\label{thm:BrauerHg}
We have
\[^c{\rm Br}(\Hcal_g)\simeq {^c{\rm Br}(k)} \oplus {\rm H}^1(k, \ZZ/\ell_g\ZZ)\oplus \ZZ/2\ZZ^{1+r_g} .\]
\end{thm}
\begin{proof}
From the description of $\Inv(\Hcal_g,\H_{\mu_\ell^{\vee}})$ it's clear that the cohomological Brauer group of $\Hcal_g$, or its prime to ${\rm char}(k)$ part in positive characteristic, is of $\ell_g$-torsion. Then the formula for $^c{\rm Br}'$ follows immediately by taking $n = \ell_g$, and the Theorem \ref{thm:EHKV} assures us that every element comes from the Brauer group.
\end{proof}

Now we want to explicitly describe the generators of $^c{\rm Br}(\Hcal_g)$. By the description in Theorem \ref{thm:BrauerHg}, an element in $^c{\rm Br}(\Hcal_g)$ must decompose as a sum of:
\begin{itemize}
    \item Elements coming from the base field.
    \item Elements coming from the cup product of the cohomology of the base field and the degree one cohomological invariant, which by \ref{lm:H1}, \ref{lm:cyclic} are all represented by cyclic algebras.
    \item The (one or two) copies of $\ZZ/2\ZZ$ coming from ${\rm B}{\rm S}_{2g+2}$ and ${\rm B}{\rm PGL}_2$.
\end{itemize}

Differently from the case of $\Mcal_{1,1}$, there are nontrivial elements which do not come from the the cup product \[{\rm H}^1(\Hcal_g,\mu_{\ell_g}) \otimes {\rm H}^1(\Hcal_g,\ZZ/\ell_g\ZZ) \rightarrow {\rm Br}(\Hcal_g),\] namely the copies of $\ZZ/2\ZZ$. One way to see this is to note that when $k$ is algebraically closed these generators are still nonzero, but \[{\rm H}^1(\Hcal_g,\mu_{\ell_g})={\rm H}^1(\Hcal_g,\ZZ/\ell_g\ZZ)=\ZZ/\ell_g\ZZ\] and the cup product ${\rm H}^1(\Hcal_g,\mu_{\ell_g}) \cdot {\rm H}^1(\Hcal_g,\ZZ/\ell_g\ZZ)$ is zero as given a generator $\gamma$ of ${\rm Inv}^1(\Hcal_g, \ZZ/\ell_g \ZZ)$ we have $\gamma \cdot \gamma = \lbrace -1 \rbrace \gamma =0$.

We now proceed to give a more detailed depiction of the elements in $^c{\rm Br}(\Hcal_g)$ that do not come neither from the base field nor from the cup product. 

Start by considering the symmetric group ${\rm S}_{2g+2}$. From the work of Schur \cite{Sch} we know that there exists a group $\Hat{S}_{2g+2}$ which is a non-split extension of ${\rm S}_{2g+2}$ by $\mu_2$, i.e. there is a non-split exact sequence of groups
\[ 1\longrightarrow \mu_2 \longrightarrow \Hat{S}_{2g+2} \longrightarrow {\rm S}_{2g+2}\longrightarrow 1. \]
The induced morphism of classifying stacks makes $\Brm \Hat{S}_{2g+2}$ into a gerbe over $\Brm {\rm S}_{2g+2}$, banded by $\mu_2$. This gerbe cannot be a trivial $\mu_2$-gerbe, as otherwise there would exists a homomorphism ${\rm S}_{2g+2}\to \Hat{S}_{2g+2}$ splitting the short exact sequence above. In particular $\Brm \Hat{S}_{2g+2}$ is not trivial over any algebraically closed field.

The cohomology group ${\rm H}^2(\Brm {\rm S}_{2g+2},\mu_2)$ is isomorphic to the group of gerbes on $\Brm {\rm S}_{2g+2}$ banded by $\mu_2$. This, together with Theorem \ref{thm:EHKV}, lets us regard $[ \Brm \Hat{S}_{2g+2}]$ as a non-trivial element of the Brauer group of $\Brm {\rm S}_{2g+2}$. 
It easily follows from Theorem \ref{prop:Inv Sn and PGL2} that $^c{\rm Br}(\Brm {\rm S}_{2g+2})\simeq \ZZ/2\ZZ$, hence $[ \Brm \Hat{S}_{2g+2} ]$ is actually a generator of this Brauer group. Moreover $[ \Brm \Hat{S}_{2g+2} ]$, this time regarded as an Azumaya algebra, cannot be a cyclic algebra, as otherwise the gerbe $\Brm \Hat{S}_{2g+2}$ should be trivial over an algebraically closed field

Let $\Hat{\Hcal}_g$ be the $\mu_2$-gerbe on $\Hcal_g$ obtained by pulling back $\Brm {\rm S}_{2g+2}$ along the classifying morphism $\Hcal_g\to \Brm {\rm S}_{2g+2}$. Then the computation contained in Theorem \ref{thm:BrauerHg} implies that $[\Hat{\Hcal}_{g}]$ generates a copy of $\ZZ/2\ZZ$ in $^c{\rm Br}(\Hcal_g)$.

Here is an alternative description of this generator as an equivalence class of a Severi-Brauer variety: by \cite{Sch} there exists a projective representation $\PP(V)$ of ${\rm S}_{2g+2}$ which does not lift to a linear representation $V$. Then the equivalence class of the Severi-Brauer variety $[\PP(V)/{\rm S}_{2g+2}]$ over $\Brm {\rm S}_{2g+2}$ is a non-trivial element of $^c{\rm Br}(\Brm {\rm S}_{2g+2})\simeq \ZZ/2\ZZ$, hence a generator. 

Again by Theorem \ref{thm:BrauerHg}, this implies that the Severi-Brauer variety over $\Hcal_g$ defined as $\Hcal_g\times_{\Brm {\rm S}_{2g+2}} [\PP(V)/{\rm S}_{2g+2}]$ generates a copy of $\ZZ/2\ZZ$ in $^c{\rm Br}(\Hcal_g)$.

We now turn to the second copy of $\ZZ/2\ZZ$ in $^c{\rm Br}(\Hcal_g)$, which only appears when $g$ is odd. Regarding $^c{\rm Br}(\Hcal_g)$ as the group of equivalence classes of Severi-Brauer varieties, it is immediate to check that a generator for the copy of $\ZZ/2\ZZ$ is given by the universal conic $\mathcal{C}\to\Hcal_g$, which is the quotient of the universal hyperelliptic curve by the universal hyperelliptic involution.

Indeed, the universal conic induces the classifying morphism $\Hcal_g\to\Brm \PGLt$, which we use to pullback the degree $2$ cohomological invariant $w_2$ of $\Brm \PGLt$.

Observe that such a universal conic exists also when $g$ is even, but in this case it is the projectivization of a vector bundle, hence its class is zero in the Brauer group.

\end{document}